\documentclass{article}
\usepackage{amssymb,amsbsy}
\usepackage{amsmath}
\usepackage{amsthm}
\usepackage{mathrsfs}
\usepackage{color}
\usepackage[dvipsnames]{xcolor}
\usepackage{tikz}

\newtheorem{theorem}{Theorem}[section]
\newtheorem{lemma}[theorem]{Lemma}
\newtheorem{proposition}[theorem]{Proposition}
\newtheorem{corollary}[theorem]{Corollary}
\newtheorem{conjecture}[theorem]{Conjecture}
\newtheorem{evidence}[theorem]{Evidence}

\theoremstyle{definition}

\newtheorem{example}[theorem]{Example}

\theoremstyle{remark}
\newtheorem{remark}[theorem]{Remark}

\begin{document}
	
	\title{A Frobenius problem suggested by prime $k$-tuplets (extended version)}
	
	\author{Aureliano M. Robles-P\'erez\thanks{Departamento de Matem\'atica Aplicada \& Instituto de Matem\'aticas (IMAG), Universidad de Granada, 18071-Granada, Spain. \newline E-mail: \textbf{arobles@ugr.es} (\textit{corresponding author}); ORCID: \textbf{0000-0003-2596-1249}.}
		\mbox{ and} Jos\'e Carlos Rosales\thanks{Departamento de \'Algebra \& Instituto de Matem\'aticas (IMAG), Universidad de Granada, 18071-Granada, Spain. \newline E-mail: \textbf{jrosales@ugr.es}; ORCID: \textbf{0000-0003-3353-4335}.} }
	
	\date{\today}
	
	\maketitle
	
	\begin{abstract}
		We study the Frobenius problem for certain $k$-tuplets, which include prime $k$-tuplets, in particular prime triplets and prime quadruplets. Moreover, we analyze some properties of the numerical semigroups associated with these tuplets.
	\end{abstract}
	
	\medskip
	
	\noindent{\it 2010 AMS Classification:} 11D07, 11N05. 	
	
	\section{Introduction}
	
	A \textit{prime triplet} is a sequence of three prime numbers $(p_1,p_2,p_3)$ such that $p_1<p_2<p_3$ and $p_3-p_1=6$. In particular, the sequences must be of the form $(p,p+2,p+6)$ or $(p,p+4,p+6)$. Since one of every three consecutive odd numbers is a multiple of three, and therefore not prime (except for three itself), we have that the prime triplets are the closest possible groupings of three prime numbers with the exceptions of $(2,3,5)$ and $(3,5,7)$.
	
	Analogously, a \textit{prime quadruplet} is a sequence $(p_1,p_2,p_3,p_4)$, of four prime numbers, such that $p_1<p_2<p_3<p_4$ and $p_4-p_1=8$. In this case, the sequences must be of the form $(p,p+2,p+6,p+8)$ and are the closest possible groupings of four prime numbers with the exceptions of $(2,3,5,7)$ and $(3,5,7,11)$.
	
	To formalize the concept of $k$-tuplet, let us follow \cite{forbes}: if $k$ is an integer greater than one, then $s(k)$ is the smallest number for which there exists a set of $k$ ordered integers $B_k=\{b_1,\ldots,b_k\}$ such that $b_1=0$, $b_k=s(k)$, and $\{b_1 \bmod q, \ldots, b_k \bmod q\} \not= \{0,1,\ldots,q-1 \}$ (that is, $\{b_1 \bmod q, \ldots, b_k \bmod q\}$ is not a complete set of residues $\bmod\ q$) for any prime number $q$. From here, a \textit{prime $k$-tuplet} is a sequence of consecutive prime numbers, $P_k=(p_1,\ldots,p_k)$, such that $p_i-p_1=b_i$ for all $i\in\{1,\ldots,k\}$ (observe that $p_k-p_1=s(k)$). Thus, roughly speaking, a prime $k$-tuplet is a sequence of $k$ consecutive prime numbers such that the difference between the first and the last (that is, $s(k)$) is as small as possible. Sometimes $B_k$ is called an \textit{admissible set} and $P_k$ is named an \textit{admissible constellations (of consecutive prime numbers)} (see \cite{erdos-riesel}).
	
	Let us observe that, if we remove the condition on the function $s(k)$, then other admissible constellations are possible. For example, we have $(p,p+4)$ (cousin primes), $(p,p+6)$ (sexy primes), $(p,p+6,p+12)$, etcetera.
	
	Let us also note that, from the definition, it is easy to see that we can only obtain a finite number of sequences of $k$ consecutive prime numbers $(p_1,\ldots,p_k)$ such that $p_k-p_1 < s(k)$. In contrast, the generalized Hardy-Littlewood conjecture (see \cite{hardy-littlewood, hardy-wright}) implies that, for every admissible set $B_k$, we can obtain infinitely many prime $k$-tuplets. In fact, all the current evidence confirm the conjecture for any admissible constellation (whether or not considering the condition on $s(k)$). 
	
	Now, let $(a_1,\ldots,a_e)$ be a sequence of positive integers which satisfies that $\gcd(a_1,\ldots,a_e)=1$. Then a classical problem in additive number theory is the Frobenius problem: what is the greatest integer $\mathrm{F}(a_1,\ldots,a_e)$ which is not an element of the set $a_1{\mathbb N}+\ldots+a_e{\mathbb N}$? Although this problem is solved for $e=2$ (see \cite{sylvester}), it is well known that it is not possible to find a polynomial formula in order to compute $\mathrm{F}(a_1,\ldots,a_e)$ if $e\geq3$ (see \cite{curtis}). Therefore, many efforts have been made to obtain partial results or to develop algorithms to get the answer of this question (see \cite{alfonsin}).
	
	Among others, the main objective of this work is the search for the solution of the Frobenius problem for prime triplets (Section~\ref{triplets}) and prime quadruplets (Section~\ref{quadruplets}). We also comment some results for $k$-tuplets with $k\geq5$ (Section~\ref{tuplets}). Moreover, from our computations, we conjecture that the Frobenius number of a $k$-tuplet is always given by a polynomial of degree two.
	
	In order to achieve our purpose, we use the theory of numerical semigroups (which is closely related with the Frobenius problem) and, in particular, the Ap\'ery set of a numerical semigroup (Section~\ref{ns}).
	
	It is convenient to note that the restriction on prime numbers is not essential to achieve our purposes. In fact, in Section~\ref{supplement} we show results for $k$-tuplets that only verify the condition on admissible constellations, that is, if $\{p_1,p_2,\ldots,p_k\}$ is a $k$-tuplet, then $\{p_1-p_1,p_2-p_1,\ldots,p_k-p_1\}$ has to be an admissible set. 

	 \section{Preliminaries (on numerical semigroups)}\label{ns}
	
	Let $\mathbb{Z}$ be the set of integers and $\mathbb{N} = \{ z\in\mathbb{Z} \mid z\geq 0\}$. A submonoid of $(\mathbb{N},+)$ is a subset $M$ of $\mathbb{N}$ such that is closed under addition and contains de zero element. A \textit{numerical semigroup} is a submonoid of $(\mathbb{N},+)$ such that $\mathbb{N}\setminus S=\{n\in\mathbb{N} \mid n\not\in S\}$ is finite.
	
	Now, let $S$ be a numerical semigroup. Since $\mathbb{N}\setminus S$ is a finite set, we can define two invariants of $S$. Namely, the \textit{Frobenius number of $S$} is the greatest integer that does not belong to $S$, denoted by $\mathrm{F}(S)$, and the \textit{genus of $S$} is the cardinality of $\mathbb{N}\setminus S$, denoted by $\mathrm{g}(S)$. Let us note that, in number theory, it is common to use the term \textit{Sylvester number} instead of the term genus (see \cite{komatsu1}).
	
	If $X$ is a non-empty subset of $\mathbb{N}$, then we denote by $\langle X \rangle$ the submonoid of $(\mathbb{N},+)$ generated by $X$, that is,
	\[ \langle X \rangle=\big\{\lambda_1x_1+\cdots+\lambda_nx_n \mid n\in\mathbb{N}\setminus \{0\}, \ x_1,\ldots,x_n\in X, \ \lambda_1,\ldots,\lambda_n\in \mathbb{N}\big\}. \]
	It is well known (see Lemma~2.1 of \cite{springer}) that $\langle X \rangle$ is a numerical semigroup if and only if $\gcd(X)=1$.
	
	If $S$ is a numerical semigroup and $S=\langle X \rangle$, then we say that $X$ is a \textit{system of generators of $S$}. Moreover, if $S\not=\langle Y \rangle$ for any subset $Y\subsetneq X$, then we say that $X$ is a \textit{minimal system of generators} of $S$. In Theorem~2.7 of \cite{springer} it is shown that each numerical semigroup admits a unique minimal system of generators and that such a system is finite. We denote by $\mathrm{msg}(S)$ the minimal system of generators of $S$. The cardinality of $\mathrm{msg}(S)$, denoted by $\mathrm{e}(S)$, is the \textit{embedding dimension} of $S$.
	
	The (extended) Frobenius problem for a numerical semigroup $S$ consists of finding formulas that allow us to compute $\mathrm{F}(S)$ and $\mathrm{g}(S)$ in terms of $\mathrm{msg}(S)$. As in the case of the Frobenius problem for sequences, such formulas are well known for $\mathrm{e}(S)=2$, but it is not possible to find polynomial formulas when $e(S)\geq3$, except for particular families of numerical semigroups.
	
	Now, let us define a very useful tool to describe a numerical semigroup $S$. If $n\in S\setminus\{0\}$, then the \textit{Ap\'ery set of $n$ in $S$} (named in honor of \cite{apery}) is $\mathrm{Ap}(S,n)=\{s\in S \mid s-n\not\in S\}$. 
	
	The following result is Lemma~2.4 of \cite{springer}.
	
	\begin{proposition}\label{prop02}
		Let $S$ be a numerical semigroup and $n\in S\setminus\{0\}$. Then the cardinality of $\mathrm{Ap}(S,n)$ is $n$. Moreover,
		\[ \mathrm{Ap}(S,n)=\{w(0)=0, w(1), \ldots, w(n-1)\}, \] 
		where $w(i)$ is the least element of $S$ congruent with $i$ modulo $n$.
	\end{proposition}
	
	The knowledge of $\mathrm{Ap}(S,n)$ allows us to solve the problem of membership of an integer to the numerical semigroup $S$. In fact, if $x\in\mathbb{Z}$, then $x\in S$ if and only if $x\geq w(x\bmod n)$. Moreover, we have the following result from \cite{brauer} (first formula) and \cite{selmer} (second one).
	
	\begin{proposition}\label{prop03}
		Let $S$ be a numerical semigroup and let $n\in S\setminus\{0\}$. Then
		\begin{enumerate}
			\item $\mathrm{F}(S)=\max(\mathrm{Ap}(S,n))-n$,
			\item $\mathrm{g}(S)=\frac{1}{n}(\sum_{w\in \mathrm{Ap}(S,n)} w)-\frac{n-1}{2}$.
		\end{enumerate}
	\end{proposition}
	
	From this proposition, it is clear that, if we have an explicit description of $\mathrm{Ap}(S,n)$, then we have the solution of the Frobenius problem for $S$.
	
	\begin{remark}\label{remark-komatsu-1}
		In \cite{komatsu1,komatsu2} the author uses Ap\'ery sets to compute the Frobenius number and the Sylvester number (that is, the genus) of numerical semigroups generated by arithmetic progressions and by arithmetic progressions with initial gaps. In particular, our corresponding results in Corollary~\ref{cor13} can be deduced from \cite{komatsu2} (see Remarks~\ref{remark-komatsu-2} and \ref{remark-komatsu-3}).
	\end{remark}
	
	\begin{remark}\label{remark-komatsu-1b}
		Different generalizations of the Frobenius number can be considered. For example, from \cite{komatsu3} (and other references therein), it is possible to define the \textit{$p$-Frobenius number} of a numerical semigroup $S=\langle a_1,a_2,\ldots,a_e\rangle$ as the greatest integer that can be represented at most $p$ ways by a linear combination with non-negative integer coefficients of $a_1,a_2,\ldots,a_e$. Obviously, the $0$-Frobenius number is the classical Frobenius number. Since in \cite{komatsu3} the author uses Ap\'ery sets to obtain his results, a possible future work would be to compute the $p$-Frobenius number for prime $k$-tuplets.
	\end{remark}
	
	Let $S$ be a numerical semigroup. Following the notation introduced in \cite{JPAA}, we say that an integer $x$ is a \textit{pseudo-Frobenius number of $S$} if $x\in \mathbb{Z}\setminus S$ and $x+s\in S$ for all $s\in S\setminus\{0\}$. We denote by $\mathrm{PF}(S)$ the set of all the pseudo-Frobenius numbers of $S$. The cardinality of $\mathrm{PF}(S)$ is an important invariant of $S$ (see \cite{barucci}) that is the so-called \textit{type} of $S$ and it is denoted by $\mathrm{t}(S)$.
	
	Let $S$ be a numerical semigroup. We define over $\mathbb{Z}$ the following binary relation: $a\leq_S b$ if $b-a\in S$. As stated in \cite{springer}, it is clear that $\leq_S$ is a non-strict partial order relation (that is, reflexive, transitive, and anti-symmetric). 
	
	The following result is Proposition~7 of \cite{froberg} (see also Proposition~2.20 of \cite{springer}) and gives us a characterization of the pseudo-Frobenius numbers in terms of the maximal elements of $\mathrm{Ap}(S,n)$ with respect to the relation $\leq_S$.
	
	\begin{proposition}\label{prop04}
		Let $S$ be a numerical semigroup and $n\in S\setminus\{0\}$. Then 
		\[\mathrm{PF}(S)=\{w-n \mid w\in \mathrm{Maximals}_{\leq_S} (\mathrm{Ap}(S,n))  \}.\]
	\end{proposition}

	\section{Prime triplets}\label{triplets}
	
	Let us recall that a prime triplet is of the form $(p,p+2,p+6)$ or of the form $(p,p+4,p+6$). In the following proposition we improve this fact.
	
	\begin{proposition}\label{prop01}
		We have that:
		\begin{enumerate}
			\item If $(p,p+2,p+6)$ is a prime triplet, then $p=6k+5$, with $k\in\mathbb{N}$.
			\item If $(p,p+4,p+6)$ is a prime triplet, then $p=6k+7$, with $k\in\mathbb{N}$.
		\end{enumerate}
	\end{proposition}
	
	\begin{proof}
		It is clear that, if $p\in\mathbb{N}\setminus\{1\}$ is not divisible by two or three, then there exists $k\in\mathbb{N}$ such that $p=6k+5$ or $p=6k+7$.
		\begin{enumerate}
			\item If $p=6k+7$, then $p+2=6k+9$ is multiple of three and can not be a prime number. Consequently, if $(p,p+2,p+6)$ is a prime triplet, then $p=6k+5$.
			\item The reasoning is similar to the above case. \qedhere
		\end{enumerate}
	\end{proof}
	
	As a consequence of the previous proposition, we have that a prime triplet is of one of the following two forms.
	\begin{enumerate}
		\item $(6k+5,6k+7,6k+11)$, with $k\in\mathbb{N}$.
		\item $(6k+7,6k+11,6k+13)$, with $k\in\mathbb{N}$.
	\end{enumerate}
	
	Since $\gcd({6k+5,6k+7,6k+11})=\gcd({6k+7,6k+11,6k+13})=1$, we can define two families of numerical semigroups ($\mathcal{T}_1$ and $\mathcal{T}_2$) associated with the prime triplets.
	\begin{itemize}
		\item $S\in\mathcal{T}_1$ if $S=\langle 6k+5,6k+7,6k+11 \rangle$, with $k\in\mathbb{N}$.
		\item $S\in\mathcal{T}_2$ if $S=\langle 6k+7,6k+11,6k+13 \rangle$, with $k\in\mathbb{N}$.
	\end{itemize}
	It is easy to check that, if $S\in\mathcal{T}_1 \cup \mathcal{T}_2$, then $\mathrm{e}(S)=3$. Moreover, from \cite{froberg}, we deduce the following result.
	
	\begin{proposition}\label{prop05}
		Let $S$ be a numerical semigroup such that $\mathrm{e}(S)=3$. Then $\mathrm{t}(S)\in\{1,2\}$. In addition, if $S=\langle n_1,n_2,n_3 \rangle$, where $n_1,n_2,n_3$ are pairwise relatively prime numbers, then $\mathrm{t}(S)=2$.
	\end{proposition}	
	
	An immediate consequence of the above comments and results is the following proposition.
	
	\begin{proposition}\label{prop06}
		If $S\in\mathcal{T}_1 \cup \mathcal{T}_2$, then $\mathrm{e}(S)=3$ and $\mathrm{t}(S)=2$.
	\end{proposition}

	\subsection{First case (family $\mathcal{T}_1$)}\label{case-one}
	
	In this subsection, we are interested in studying the numerical semigroups of the form $S=\langle 6k+5,6k+7,6k+11 \rangle$, where $k\in\mathbb{N}$.
	
	Straightforward computations lead to the following result.
	
	\begin{lemma}\label{lem07}
		If $k\in\mathbb{N}$, then we have the equalities
		\begin{enumerate}
			\item $3(6k+7)=2(6k+5)+1(6k+11)$;
			\item $(2k+2)(6k+11)=(2k+3)(6k+5)+1(6k+7)$;
			\item $2(6k+7)+(2k+1)(6k+11)=(2k+5)(6k+5)$.
		\end{enumerate}
	\end{lemma}
	
	Let us now see the key result of this subsection.
	
	\begin{theorem}\label{thm08}
		If $k\in\mathbb{N}$ and $S=\langle 6k+5,6k+7,6k+11 \rangle$, then
		\[ \mathrm{Ap}(S,6k+5)=\left\{ a(6k+7)+b(6k+11) \mid (a,b)\in C \right\}, \]
		where $ C = \big( \{0,1,2\} \times \{0,1,\ldots,2k+1\} \big) \setminus \{ (2,2k+1)\} $.
	\end{theorem}
	
	\begin{proof}
		From Lemma~\ref{lem07}, we easily deduce that $\mathrm{Ap}(S,6k+5)$ is a subset of $\left\{a(6k+7)+b(6k+11) \mid (a,b)\in C \right\}$. Now then, since the cardinality of $C$ is less than or equal to $3(2k+2)-1=6k+5$, and taking into consideration Proposition~\ref{prop02}, we get that $\mathrm{Ap}(S,6k+5)=\left\{a(6k+7)+b(6k+11) \mid (a,b)\in C \right\}$.
	\end{proof}
	
	Let us observe that we can rewrite several elements of $\mathrm{Ap}(S,6k+5)$ as follows.
	\begin{itemize}
		\item $(6k+7)+b(6k+11)=(b+1)(6k+11)-4$, for all $b\in\{0,1,\ldots,2k+1\}$.
		\item $2(6k+7)+b(6k+11)=(b+2)(6k+11)-8$, for all $b\in\{0,1,\ldots,2k\}$.
	\end{itemize}
	Thus, we can easily describe the Apéry set by arranging its elements in increasing order.
	
	\begin{corollary}\label{cor08}
		If $k\in\mathbb{N}$ and $S=\langle 6k+5,6k+7,6k+11 \rangle$, then
		\[ \mathrm{Ap}(S,6k+5)=\{0; (6k+11)-4,6k+11; 2(6k+11)-8,2(6k+11)-4,2(6k+11); \]
		\[ \ldots; (2k+1)(6k+11)-8,(2k+1)(6k+11)-4,(2k+1)(6k+11); \]
		\[ (2k+2)(6k+11)-8,(2k+2)(6k+11)-4 \}. \]
	\end{corollary}

	In this way we can identify a certain pattern. In fact, in the previous corollary, we have used ``;'' to separate certain groups of numbers.
	
	As mentioned in Section~\ref{ns}, knowledge of the Ap\'ery set allows us to obtain information about the numerical semigroup. Thus, in the current case we have the following result. (Let us observe that, from Corollary~\ref{cor09}, we recover Proposition~\ref{prop06}.)
	
	\begin{corollary}\label{cor09}
		If $k\in\mathbb{N}$ and $S=\langle 6k+5,6k+7,6k+11 \rangle$, then
		\begin{enumerate}
			\item $\mathrm{PF}(S) = \{ 12k^2+28k+9,12k^2+28k+13 \}$;
			\item $\mathrm{F}(S) = 12k^2+28k+13$;
			\item $\mathrm{g}(S) = 6k^2+16k+8$.
		\end{enumerate}
	\end{corollary}
	
	\begin{proof}
		\begin{enumerate}
			\item By Theorem~\ref{thm08}, we have that $\mathrm{Maximals}_{\leq_S} (\mathrm{Ap}(S,6k+5)) = \{(6k+7)+(2k+1)(6k+11), 2(6k+7)+2k(6k+11)\}$. Thereby, from Proposition~\ref{prop04}, we can assert that
			$\mathrm{PF}(S) = \{(6k+7)+(2k+1)(6k+11)-(6k+5), 2(6k+7)+2k(6k+11)-(6k+5)\} = \{ 12k^2+28k+13,12k^2+28k+9 \}$.
			
			\item It is clear that $\mathrm{F}(S) = \max(\mathrm{PF}(S))$ and, therefore, $\mathrm{F}(S) = 12k^2+28k+13$.
			
			\item From Theorem~\ref{thm08}, we have that
			\[ \mathrm{Ap}(S,6k+5) = \{ 0,6k+11,\ldots,(1+2k)(6k+11), \]
			\[ 6k+7,(6k+7)+(6k+11),\ldots,(6k+7)+(1+2k)(6k+11), \]
			\[ 2(6k+7),2(6k+7)+(6k+11),\ldots,2(6k+7)+2k(6k+11) \}. \]
			Now, by Proposition~\ref{prop03}, we get that
			\[\mathrm{g}(S) = \frac{1}{6k+5} \bigg(6k+11+\cdots+(1+2k)(6k+11)+6k+7+(6k+7)+(6k+11)+ \]
			\[ \cdots+(6k+7)+(1+2k)(6k+11)+2(6k+7)+2(6k+7)+(6k+11)+ \]
			\[ \cdots+2(6k+7)+2k(6k+11) \bigg) - \frac{(6k+5)-1}{2} = 6k^2+16k+8.\] \qedhere 
		\end{enumerate}
	\end{proof}
	
	\begin{remark}\label{case1}
		From Corollary~\ref{cor09}, we deduce that, if $p=6k+5$ with $k\in\mathbb{N}$ (in particular, if $(p,p+2,p+6)$ is a prime triplet), then $\mathrm{F}(p,p+2,p+6)=\frac{p^2+4p-6}{3}$.
	\end{remark}
	
	We finish with an illustrative example.
	
	\begin{example}\label{exmp10}
		If $k=1$, then $S=\langle 11,13,17 \rangle$, $\mathrm{PF}(S) = \{ 49,53 \}$, $\mathrm{F}(S) = 53$, and $\mathrm{g}(S) = 30$ (by Corollary~\ref{cor09}). Moreover, from Corollary~\ref{cor08}, we know that $\mathrm{Ap}(S,11) = \{0,13,17,26,30,34,43,47,51,60,64\}$.
	\end{example}

	\subsection{Second case (family $\mathcal{T}_2$)}\label{case-two}
	
	Now we study numerical semigroups of the form $S=\langle 6k+7,6k+11,6k+13 \rangle$, where $k\in\mathbb{N}$.
	
	The results are similar to those of Subsection~\ref{case-one}. Therefore, we omit the proofs.
	
	\begin{lemma}\label{lem11}
		If $k\in\mathbb{N}$, then we have the equalities
		\begin{enumerate}
			\item $3(6k+11)=1(6k+7)+2(6k+13)$;
			\item $(2k+3)(6k+13)=(2k+4)(6k+7)+1(6k+11)$;
			\item $2(6k+11)+(2k+1)(6k+13)=(2k+5)(6k+7)$.
		\end{enumerate}
	\end{lemma}
	
	\begin{theorem}\label{thm12}
		If $k\in\mathbb{N}$ and $S=\langle 6k+7,6k+11,6k+13 \rangle$, then
		\[ \mathrm{Ap}(S,6k+7)=\left\{ a(6k+11)+b(6k+13) \mid (a,b)\in C \right\}, \]
		where	$ C = \big( \{0,1,2\} \times \{0,1,\ldots,2k+2\} \big) \setminus \{ (2,2k+1),(2,2k+2)\} $.
	\end{theorem}
	
	Let us observe that
	\begin{itemize}
		\item $(6k+11)+b(6k+13)=(b+1)(6k+13)-2$, for all $b\in\{0,1,\ldots,2k+2\}$;
		\item $2(6k+11)+b(6k+13)=(b+2)(6k+11)-4$, for all $b\in\{0,1,\ldots,2k\}$.
	\end{itemize}
	From here, we describe the Apéry set by arranging its elements in increasing order.

	\begin{corollary}\label{cor12}
		If $k\in\mathbb{N}$ and $S=\langle 6k+7,6k+11,6k+13 \rangle$, then
		\[ \mathrm{Ap}(S,6k+7)=\{0; (6k+13)-2,6k+13; 2(6k+13)-4,2(6k+13)-2,2(6k+13); \]
		\[ \ldots; (2k+2)(6k+13)-4,(2k+2)(6k+13)-2,(2k+2)(6k+13); (2k+3)(6k+13)-2 \}. \]
	\end{corollary}
	
	\begin{corollary}\label{cor13}
		If $k\in\mathbb{N}$ and $S=\langle 6k+7,6k+11,6k+13 \rangle$, then
		\begin{enumerate}
			\item $\mathrm{PF}(S) = \{ 12k^2+32k+15,12k^2+38k+30 \}$;
			\item $\mathrm{F}(S) = 12k^2+38k+30$;
			\item $\mathrm{g}(S) = 6k^2+20k+16$.
		\end{enumerate}
	\end{corollary}
	
	\begin{remark}\label{case2}
		From Corollary~\ref{cor13}, we have that, if $p=6k+7$ with $k\in\mathbb{N}$ (in particular, if $(p,p+4,p+6)$ is a prime triplet), then $\mathrm{F}(p,p+4,p+6)=\frac{p^2+5p+6}{3}$.
	\end{remark}
	
	Let us see an example.
	
	\begin{example}\label{exmp14}
		If $k=0$, then $S=\langle 7,11,13 \rangle$, $\mathrm{PF}(S) = \{ 15,30 \}$, $\mathrm{F}(S) = 30$, and $\mathrm{g}(S) = 16$ (by Corollary~\ref{cor13}). Moreover, from Corollary~\ref{cor12}, we get that $\mathrm{Ap}(S,7) = \{0,11,13,22,24,26,37\}$.
	\end{example}
	
	\begin{remark}\label{remark-komatsu-2}
		The values of $\mathrm{F}(p,p+4,p+6)$ and $\mathrm{g}(p,p+4,p+6)$, for $p=6k+7$ with $k\in\mathbb{N}$, can be obtained from Theorems~2 and 3 of \cite{komatsu2}, respectively. In fact, it is enough to consider $a=p$, $K=1$, $k=3$, $d=2$, and $r=2$ in those theorems.
	\end{remark}
	
	\begin{remark}\label{remark-komatsu-3}
		Let us recall that the \textit{Sylvester sum} of a numerical semigroup $S=\langle a_1,a_2,\ldots,a_e\rangle$ is the value
		\[ \mathrm{s}(S) = \sum_{x\in\mathbb{N}\setminus S}x. \]
		In \cite{komatsu2} the author computes the Sylvester sum of numerical semigroups generated by arithmetic progressions with initial gaps. To get the result, he again uses the Ap\'ery set $\mathrm{Ap}(A,a_1)=\{w_0=0,w_1,\ldots,w_{a_1-1}\}$ by the formula
		\begin{equation}\label{Sylv-sum}
			\mathrm{s}(a_1,a_2,\ldots,a_e) = \mathrm{s}(S) = \frac{1}{2a_1}\sum_{i=1}^{a_1-1} w_i^2 - \frac{1}{2}\sum_{i=1}^{a_1-1} w_i + \frac{a_1^2-1}{12}.
		\end{equation}
		As a consequence, by Theorem~1 of \cite{komatsu2} we have that, if $p=6k+7$ with $k\in\mathbb{N}$, then $\mathrm{s}(p,p+4,p+6)=\frac{1}{108}(2p^4+24p^3+93p^2+202p+435)$. Of course, we could directly apply \eqref{Sylv-sum} to the Ap\'ery sets of the family $\mathcal{T}_2$ and obtain the same result. Even more, we could consider \eqref{Sylv-sum} for all the families we analyze in this work, but we prefer to leave it to the reader as a computation exercise.
	\end{remark}

	\section{Prime quadruplets}\label{quadruplets}
	
	It is well known that a prime quadruplet is of the form $(p,p+2,p+6,p+8)$. In the following result we improve this expression.
	
	\begin{proposition}\label{prop15}
		We have that, if $(p,p+2,p+6,p+8)$ is a prime quadruplet, then either $p=4k+5$ or $p=4k+7$ for some $k\in\mathbb{N}$.
	\end{proposition}
	
	\begin{proof}
		It is clear that, if $p\in\mathbb{N}\setminus\{1\}$, then there exists $k\in\mathbb{N}$ such that $p=4k+i$, with $i\in\{0,1,2,3\}$. But, since $p, p+2, p+6, p+8$ are prime numbers, then $i\not=0$ and $i\not=2$. In addition, let us note that $(1,3,7,9)$ and $(3,5,9,11)$ are not prime quadruplets.
	\end{proof}
	
	In contrast to the case of prime triplets, the prime quadruplets are giving by a unique form. However, to solve the Frobenius problem we have to consider two different expressions.
	\begin{itemize}
		\item $(4k+5,4k+7,4k+11,4k+13)$, with $k\in\mathbb{N}$.
		\item $(4k+7,4k+9,4k+13,4k+15)$, with $k\in\mathbb{N}$.
	\end{itemize}
	
	Now, since $\gcd(4k+5,4k+7)=\gcd(4k+7,4k+9)=1$ for all $k\in\mathbb{N}$, we have two families of numerical semigroups ($\mathcal{Q}_1$ and $\mathcal{Q}_2$) related to the prime quadruplets.
	\begin{itemize}
		\item $S\in\mathcal{Q}_1$ if $S=\langle 4k+5,4k+7,4k+11,4k+13 \rangle$, with $k\in\mathbb{N}$.
		\item $S\in\mathcal{Q}_2$ if $S=\langle 4k+7,4k+9,4k+13,4k+15 \rangle$, with $k\in\mathbb{N}$.
	\end{itemize}
	Let us observe that $\langle 5,7,11,13 \rangle$ and $\langle 7,9,13,15 \rangle$ are numerical semigroups with embedding dimension equal to four. Moreover, since $4k+13<2(4k+5)$ and $4k+15<2(4k+7)$ for all $k\geq 1$, we can assert that $e(S)=4$ for every $S\in\mathcal{Q}_1\cup\mathcal{Q}_2$.
	
	\begin{remark}\label{rem16a}
		We have $S=\langle 1,3,7,9 \rangle = \langle 1 \rangle$ and $S=\langle 3,5,9,11 \rangle =\langle 3,5,11 \rangle$. Thus, $\mathrm{e}(S)<4$ in both cases. This is another reason for eliminating  $(1,3,7,9)$ and $(3,5,9,11)$ as possibilities in Proposition~\ref{prop15}.  
	\end{remark}
	
	\begin{remark}\label{rem16b}
		It is possible to improve Proposition~\ref{prop15} a little more. In fact, if we assume that the elements of the quadruplet are not divisible by two or three, then we have that either $p=12k+5$ or $p=12k+11$ for some $k\in\mathbb{N}$. Even more, if we consider that the elements are not divisible by two, three or five, then all prime quadruplets (except $(5,7,11,13)$) are of the form $p=30k+11$ for some $k\in\mathbb{N}$. In any case, we obtain essentially the same conclusions for the families $\mathcal{Q}_1$ and $\mathcal{Q}_2$ associated with $p=12k+5$ and $p=12k+11$.
	\end{remark}

	\subsection{First expression (family $\mathcal{Q}_1$)}
	
	Let $S$ be a numerical semigroup of the form $S=\langle 4k+5,4k+7,4k+11,4k+13 \rangle$, where $k\in\mathbb{N}$. Again we omit the proofs of the results because they are similar to those of Subsection~\ref{case-one}.
	
	\begin{remark}
		The numerical semigroup $S=\langle 5,7,11,13 \rangle$ behaves somewhat  differently from the rest of the cases. In fact, $\mathrm{Ap}(S,5) = \{0,7,11,13,14\}$ and
		\begin{enumerate}
			\item $\mathrm{PF}(S) = \{ 6,8,9 \}$;
			\item $\mathrm{F}(S) = 9$;
			\item $\mathrm{g}(S) = 7$;
			\item $\mathrm{t}(S) = 3$.
		\end{enumerate}
		It is clear that, except the value of $\mathrm{g}(S)$, these results are not obtained by Corollary~\ref{cor19}.
	\end{remark}
	
	\begin{lemma}\label{lem17}
		If $k\in\mathbb{N}\setminus\{0\}$, then we have the equalities
		\begin{enumerate}
			\item $3(4k+7)=2(4k+5)+1(4k+11)$;
			\item $3(4k+11)=1(4k+7)+2(4k+13)$;
			\item $(k+2)(4k+13)=(k+3)(4k+5)+1(4k+11)$;
			\item $1(4k+7)+1(4k+11)=1(4k+5)+1(4k+13)$;
			\item $1(4k+7)+(k+1)(4k+13)=(k+4)(4k+5)$;
			\item $2(4k+7)+1(4k+13)=1(4k+5)+2(4k+11)$;
			\item $1(4k+11)+(k+1)(4k+13)=(k+2)(4k+5)+2(4k+7)$;
			\item $2(4k+11)+k(4k+13)=(k+3)(4k+5)+1(4k+7)$.
		\end{enumerate}
	\end{lemma}
	
	\begin{theorem}\label{thm18}
		If $k\in\mathbb{N}\setminus\{0\}$ and $S=\langle 4k+5,4k+7,4k+11,4k+13 \rangle$, then
		\[ \mathrm{Ap}(S,4k+5)=\left\{ a(4k+7)+b(4k+11)+c(4k+13) \mid (a,b,c)\in C \right\}, \]
		where	$ C = C_a \cup C_b \cup C_c$, with $C_a = \big( \{1\} \times \{0\} \times \{0,1,\ldots,k\} \big) \cup \{ (2,0,0)\} $, $C_b = (\{0\} \times \{1,2\} \times \{0,1,\ldots,k\}) \setminus \{(0,2,k)\} $, and $C_c = \{0\} \times \{0\} \times \{0,1,\ldots,k+1\} $.
	\end{theorem}
	
	Let us observe that
	\begin{itemize}
		\item $(4k+7) + c(4k+13) = (c+1)(4k+13) - 6$, for all $c\in\{0,1,\ldots,k\}$;
		\item $(4k+11) + c(4k+13) = (c+1)(4k+13) - 2$, for all $c\in\{0,1,\ldots,k\}$;
		\item $2(4k+11) + c(4k+13) = (c+2)(4k+15) - 4$, for all $c\in\{0,1,\ldots,k-1\}$.
	\end{itemize}
	
	Therefore, we can give the Apéry set by arranging its elements in increasing order.
	
	\begin{corollary}\label{cor18}
		If $k\in\mathbb{N}\setminus\{0\}$ and $S=\langle 4k+5,4k+7,4k+11,4k+13 \rangle$, then
		\[ \mathrm{Ap}(S,4k+5)=\{0; (4k+13)-6,(4k+13)-2,4k+13; 2(4k+7); \]
		\[ 2(4k+13)-6,2(4k+13)-4,2(4k+13)-2,2(4k+13); \ldots; \]
		\[ (k+1)(4k+13)-6,(k+1)(4k+13)-4,(k+1)(4k+13)-2, (k+1)(4k+13) \}. \]
	\end{corollary}
	
	\begin{corollary}\label{cor19}
		If $k\in\mathbb{N}\setminus\{0\}$ and $S=\langle 4k+5,4k+7,4k+11,4k+13 \rangle$, then
		\begin{enumerate}
			\item $\mathrm{PF}(S) = \{ 4k+9,4k^2+13k+2,4k^2+13k+4,4k^2+13k+6,4k^2+13k+8 \}$;
			\item $\mathrm{F}(S) = 4k^2+13k+8$;
			\item $\mathrm{g}(S) = 2k^2+8k+7$;
			\item $\mathrm{t}(S) = 5$.
		\end{enumerate}
	\end{corollary}
	
	\begin{remark}\label{case11}
		By Corollary~\ref{cor19}, if $p=4k+5$ for some $k\in\mathbb{N}\setminus\{0\}$ (in particular, if $(p,p+2,p+6,p+8)$ is a prime quadruplet with $p=4k+5$ and $k\geq1$), then $\mathrm{F}(p,p+2,p+6,p+8)=\frac{p^2+3p-8}{4}$.
	\end{remark}
	
	We finish with an illustrative example.
	
	\begin{example}\label{exmp20}
		For $k=24$ we have $S=\langle 101,103,107,109 \rangle$ and, by Corollary~\ref{cor19}, $\mathrm{PF}(S) = \{ 105,2618,2620,2622,2624 \}$, $\mathrm{F}(S) = 2624$, and $\mathrm{g}(S) = 1351$. Moreover, by Corollary~\ref{cor18}, $\mathrm{Ap}(S,11) = \{0,103,107,109,206,212,214,216,218,$ $\ldots,2719,2721,2723,2725\}$.
	\end{example}

	\subsection{Second expression (family $\mathcal{Q}_2$)}
	
	Let $S$ be a numerical semigroup of the form $S=\langle 4k+7,4k+9,4k+13,4k+15 \rangle$, where $k\in\mathbb{N}$. 
	
	\begin{lemma}\label{lem21}
		If $k\in\mathbb{N}$, then we have the equalities
		\begin{enumerate}
			\item $3(4k+9)=2(4k+7)+1(4k+13)$;
			\item $3(4k+13)=1(4k+9)+2(4k+15)$;
			\item $(k+2)(4k+15)=(k+3)(4k+7)+1(4k+9)$;
			\item $1(4k+9)+1(4k+13)=1(4k+7)+1(4k+15)$;
			\item $2(4k+9)+1(4k+15)=1(4k+7)+2(40k+13)$;
			\item $1(4k+13)+(k+1)(4k+15)=(k+4)(4k+7)$.
		\end{enumerate}
	\end{lemma}
	
	\begin{theorem}\label{thm22}
		If $k\in\mathbb{N}$ and $S=\langle 4k+7,4k+9,4k+13,4k+15 \rangle$, then
		\[ \mathrm{Ap}(S,4k+7)=\left\{ a(4k+9)+b(4k+13)+c(4k+15) \mid (a,b,c)\in C \right\}, \]
		where	$ C = C_a \cup C_b \cup C_c$, with $C_a = \big( \{1\} \times \{0\} \times \{0,1,\ldots,k+1\} \big) \cup \{ (2,0,0)\} $, $C_b = \{0\} \times \{1,2\} \times \{0,1,\ldots,k\} $, and $C_c = \{0\} \times \{0\} \times \{0,1,\ldots,k+1\} $.
	\end{theorem}
	
	Let us observe that
	\begin{itemize}
		\item $(4k+9) + c(4k+15) = (c+1)(4k+15) - 6$, for all $c\in\{0,1,\ldots,k+1\}$;
		\item $(4k+13) + c(4k+15) = (c+1)(4k+15) - 2$, for all $c\in\{0,1,\ldots,k\}$;
		\item $2(4k+13) + c(4k+15) = (c+2)(4k+15) - 4$, for all $c\in\{0,1,\ldots,k\}$.
	\end{itemize}
	Consequently, we can arrange the elements of the Ap\'ery set in increasing order.
	
	\begin{corollary}\label{cor22}
		If $k\in\mathbb{N}$ and $S=\langle 4k+7,4k+9,4k+13,4k+15 \rangle$, then
		\[ \mathrm{Ap}(S,4k+7)=\{0; (4k+15)-6,(4k+15)-2,4k+15; 2(4k+9); \]
		\[ 2(4k+15)-6,2(4k+15)-4,2(4k+15)-2,2(4k+15); \ldots; \]
		\[ (k+1)(4k+15)-6,(k+1)(4k+15)-4,(k+1)(4k+15)-2,(k+1)(4k+15); \]
		\[ (k+2)(4k+15)-6,(k+2)(4k+15)-4 \}. \]
	\end{corollary}
	
	\begin{remark}
		If $S=\langle 7,9,13,15 \rangle$, then $\mathrm{Ap}(S,7)=\{0,9,13,15,18,24,26\}$. Therefore, we can apply Corollary~\ref{cor22} if we only consider the five first values and the two last.
	\end{remark}
	
	\begin{corollary}\label{cor23}
		If $k\in\mathbb{N}$ and $S=\langle 4k+7,4k+9,4k+13,4k+15 \rangle$, then
		\begin{enumerate}
			\item $\mathrm{PF}(S) = \{ 4k+11, 4k^2+19k+17, 4k^2+19k+19 \}$;
			\item $\mathrm{F}(S) = 4k^2+19k+19$;
			\item $\mathrm{g}(S) = 2k^2+10k+12$;
			\item $\mathrm{t}(S) = 3$.
		\end{enumerate}
	\end{corollary}
	
	\begin{remark}\label{case12}
		By Corollary~\ref{cor23}, if $p=4k+7$ for some $k\in\mathbb{N}$ (in particular, if $(p,p+2,p+6,p+8)$ is a prime quadruplet with $p=4k+7$ and $k\geq0$), then $\mathrm{F}(p,p+2,p+6,p+8)=\frac{p^2+5p-8}{4}$.
	\end{remark}
	
	We finish with an illustrative example.
	
	\begin{example}\label{exmp24}
		For $k=1$ we have $S=\langle 11,13,17,19 \rangle$, $\mathrm{PF}(S) = \{ 15,40,42 \}$, $\mathrm{F}(S) = 42$, and $\mathrm{g}(S) = 24$ (by Corollary~\ref{cor23}). Moreover, from Corollary~\ref{cor22}, we know that $\mathrm{Ap}(S,11) = \{0,13,17,19,26,32,34,36,38,51,53\}$.
	\end{example}

	\section{Prime $k$-tuplets}\label{tuplets}
	
	From the contents of Sections~\ref{triplets} and \ref{quadruplets}, it looks like that the problem is going to became more and more longueur as soon as we consider larger and larger $k$-tuplets. Anyway, it is not difficult to see what happens when $k\in\{5,6,7,8\}$ and, in this way, to propose a conjecture (see Section~\ref{supplement}).
	
	There exist two families of prime quintuplets. Namely, $(p,p+2,p+6,p+8,p+12)$ and $(p,p+4,p+6,p+10,p+12)$. Now, if $k\in\mathbb{N}$, then these families are associated with the values $6k+5$ and $6k+7$, respectively. In addition, the Frobenius number is equal to
	\begin{itemize}
		\item $\frac{p^2+7p-12}{6}$ for $(p,p+2,p+6,p+8,p+12)$ with $p=6k+5\geq11$,
		\item $\frac{p^2+11p+12}{6}$ for $(p,p+4,p+6,p+10,p+12)$ with $p=6k+7\geq7$.
	\end{itemize}
	Furthermore, if we consider numerical semigroups generated by prime quintuplets, then the type is given by
	\begin{itemize}
		\item $\mathrm{t}(S)=6$ for $S=\langle p,p+2,p+6,p+8,p+12 \rangle$ with $p=6k+11\geq11$,
		\item $\mathrm{t}(S)=4$ for $S=\langle p,p+4,p+6,p+10,p+12 \rangle$ with $p=6k+7\geq13$.
	\end{itemize}
	
	Although the expression $(p,p+4,p+6,p+10,p+12,p+16)$ is the unique possibility for the prime sextuplets, we have to consider four different families to study the Frobenius problem. Such families correspond to the values $p=8k+r$ with $k\in\mathbb{N}$ and $r\in\{7,9,11,13\}$. In this case, we have that the Frobenius number is equal to
	\begin{itemize}
		\item $\frac{p^2+9p+16}{8}$ for $p=8k+7$,
		\item $\frac{p^2+15p+16}{8}$ for $p=8k+9$,
		\item $\frac{p^2+13p+16}{8}$ for $p=8k+11$,
		\item $\frac{p^2+11p+16}{8}$ for $p=8k+13$.
	\end{itemize}
	Moreover, for the four families of numerical semigroups associated with the prime sextuplets, the type is equal to $9$ (of course, for $p=8k+7$ with $k\geq1$), $5$, $5$, and $7$, respectively. It is interesting to observe that, although the value of the type is the same in the second and third cases, the structures of their sets of pseudo-Frobenius numbers are essentially different (see Subsection~\ref{sextuplets}).
	
	For prime septuplets we have two families: $(p,p+2,p+6,p+8,p+12,p+18,p+20)$ and $(p,p+2,p+8,p+12,p+14,p+18,p+20)$. Now, if $k\in\mathbb{N}$, then $p=10k+11$ for the first case and $p=10k+19$ for the second one. Thus, the pair (Frobenius number, type) is equal to $(\frac{p^2+9p-20}{10},13)$ (if $p\geq31$) and $(\frac{p^2+11p-20}{10},11)$ (if $p\geq19$), respectively.
	
	Finally, there are three families of prime octuplets: $(p,p+2,p+6,p+12,p+14,p+20,p+24,p+26)$, $(p,p+2,p+6,p+8,p+12,p+18,p+20,p+26)$, and $(p,p+6,p+8,p+14,p+18,p+20,p+24,p+26)$. As discussed in Subsection~\ref{octuplets}, we have to study $39$ cases ($13$ for each family) to obtain formulas for the Frobenius number. 
	
	From all these comments, we establish the following conjecture.
	
	\begin{conjecture}
		Let us consider a prime $k$-tuplet with first element $p$ and last element $p+q$. Then we have that:
		\begin{enumerate}
			\item The Frobenius number is given by a quadratic polynomial $a_2p^2+a_1p+a_0$.
			\item The leading coefficient is $a_2=\frac{2}{q}$.
			\item The constant term $a_0$ is an integer.
		\end{enumerate}
	\end{conjecture}
	
	Let us observe that we have got $a_0=\pm2$ when $1\leq k\leq7$. However, the values $a_0=-4$, $a_0=-6$, and $a_0=10$ appear in prime octuplets. Thus, without going into further details, we have preferred to state that $a_0$ is an integer.

	\section{Supplement}\label{supplement}
	
	In this section we give the statements of the results for triplets are not necessarily prime triplets (note that we have studied all possibilities for quadruplets in Section~\ref{quadruplets}). Moreover, we also present analogues results for (prime) quintuplets and (prime) sextuplets and we show computational evidence for (prime) septuplets and (prime) octuplets.
	
	\subsection{Other triplets}
	
	\subsubsection{Other triplets (1): $S=\langle 6k+3,6k+5,6k+9 \rangle$}
	
	Note that $\gcd(6k+3,6k+9)=3$. Therefore, $S=\langle 6k+3,6k+5,6k+9 \rangle$ is related to $T=\langle 2k+1,6k+5,2k+3 \rangle = \langle 2k+1,2k+3 \rangle$.
	
	\begin{lemma}\label{lem11-1}
		If $k\in\mathbb{N}$, then we have the equalities
		\begin{enumerate}
			\item $3(6k+5)=2(6k+3)+1(6k+9)$;
			\item $(2k+1)(6k+9)=(2k+3)(6k+3)+0(6k+5)$.
		\end{enumerate}
	\end{lemma}
	
	\begin{theorem}\label{thm12-1}
		If $k\in\mathbb{N}$ and $S=\langle 6k+3,6k+5,6k+9 \rangle$, then
		\[ \mathrm{Ap}(S,6k+3)=\left\{ a(6k+5)+b(6k+9) \mid (a,b)\in C \right\}, \]
		where $ C = \{0,1,2\} \times \{0,1,\ldots,2k\} $.
	\end{theorem}
	
	Let us observe that
	\begin{itemize}
		\item $(6k+5)+b(6k+9)=(b+1)(6k+9)-4$, for all $b\in\{0,1,\ldots,2k\}$;
		\item $2(6k+5)+b(6k+9)=(b+2)(6k+9)-8$, for all $b\in\{0,1,\ldots,2k\}$.
	\end{itemize}
	
	\begin{corollary}\label{cor12-1}
		If $k\in\mathbb{N}$ and $S=\langle 6k+3,6k+5,6k+9 \rangle$, then
		\[ \mathrm{Ap}(S,6k+3)=\{0; (6k+9)-4,6k+9; 2(6k+9)-8,2(6k+9)-4,2(6k+9); \]
		\[ \ldots; 2k(6k+9)-8,2k(6k+9)-4,2k(6k+9); (2k+1)(6k+9)-8,(2k+1)(6k+9)-4; \]
		\[ (2k+2)(6k+9)-8 \}. \]
	\end{corollary}
	
	\begin{corollary}\label{cor13-1}
		If $k\in\mathbb{N}$ and $S=\langle 6k+3,6k+5,6k+9 \rangle$, then
		\begin{enumerate}
			\item $\mathrm{PF}(S) = \{ 12k^2+24k+7 \}$;
			\item $\mathrm{F}(S) = 12k^2+24k+7$;
			\item $\mathrm{g}(S) = 6k^2+12k+4$.
		\end{enumerate}
	\end{corollary}
	
	\begin{remark}\label{casetri-1-1}
		If $p=6k+3$ with $k\in\mathbb{N}$, then $\mathrm{F}(p,p+2,p+6)=\frac{p^2+6p-6}{3}$.
	\end{remark}

	\subsubsection{Other triplets (2): $S=\langle 6k+3,6k+7,6k+9 \rangle$}
	
	Note that $\gcd(6k+3,6k+9)=3$. Therefore, $S=\langle 6k+3,6k+7,6k+9 \rangle$ is related to $T=\langle 2k+1,6k+7,2k+3 \rangle = \langle 2k+1,2k+3 \rangle$.
	
	\begin{lemma}\label{lem11-2}
		If $k\in\mathbb{N}$, then we have the equalities
		\begin{enumerate}
			\item $3(6k+7)=1(6k+3)+2(6k+9)$;
			\item $(2k+1)(6k+9)=(2k+3)(6k+3)+0(6k+7)$.
		\end{enumerate}
	\end{lemma}
	
	\begin{theorem}\label{thm12-2}
		If $k\in\mathbb{N}$ and $S=\langle 6k+3,6k+7,6k+9 \rangle$, then
		\[ \mathrm{Ap}(S,6k+3)=\left\{ a(6k+7)+b(6k+9) \mid (a,b)\in C \right\}, \]
		where $ C = \{0,1,2\} \times \{0,1,\ldots,2k\} $.
	\end{theorem}
	
	Let us observe that
	\begin{itemize}
		\item $(6k+7)+b(6k+9)=(b+1)(6k+9)-2$, for all $b\in\{0,1,\ldots,2k\}$;
		\item $2(6k+7)+b(6k+9)=(b+2)(6k+9)-4$, for all $b\in\{0,1,\ldots,2k\}$.
	\end{itemize}
	
	\begin{corollary}\label{cor12-2}
		If $k\in\mathbb{N}$ and $S=\langle 6k+3,6k+7,6k+9 \rangle$, then
		\[ \mathrm{Ap}(S,6k+3)=\{0; (6k+9)-2,6k+9; 2(6k+9)-4,2(6k+9)-2,2(6k+9); \ldots; \]
		\[ 2k(6k+9)-4,2k(6k+9)-2,2k(6k+9); (2k+1)(6k+9)-4,(2k+1)(6k+9)-2; \]
		\[ (2k+2)(6k+9)-4 \}. \]
	\end{corollary}
	
	\begin{corollary}\label{cor13-2}
		If $k\in\mathbb{N}$ and $S=\langle 6k+3,6k+7,6k+9 \rangle$, then
		\begin{enumerate}
			\item $\mathrm{PF}(S) = \{ 12k^2+24k+11 \}$;
			\item $\mathrm{F}(S) = 12k^2+24k+11$;
			\item $\mathrm{g}(S) = 6k^2+12k+6$.
		\end{enumerate}
	\end{corollary}
	
	\begin{remark}\label{casetri2-1}
		If $p=6k+3$ with $k\in\mathbb{N}$, then $\mathrm{F}(p,p+4,p+6)=\frac{p^2+6p+6}{3}$.
	\end{remark}

	\subsubsection{Other triplets (3): $S=\langle 6k+7,6k+9,6k+13 \rangle$}
	
	Since $\gcd({6k+7,6k+9})=1$, we can assert that $S=\langle 6k+7,6k+9,6k+13 \rangle$ is a numerical semigroup.
	
	\begin{lemma}\label{lem11-3}
		If $k\in\mathbb{N}$, then we have the equalities
		\begin{enumerate}
			\item $3(6k+9)=2(6k+7)+1(6k+13)$;
			\item $(2k+3)(6k+13)=(2k+3)(6k+7)+2(6k+9)$;
			\item $1(6k+9)+(2k+2)(6k+13)=(2k+5)(6k+7)$.
		\end{enumerate}
	\end{lemma}
	
	\begin{theorem}\label{thm12-3}
		If $k\in\mathbb{N}$ and $S=\langle 6k+7,6k+9,6k+13 \rangle$, then
		\[ \mathrm{Ap}(S,6k+7)=\left\{ a(6k+9)+b(6k+13) \mid (a,b)\in C \right\}, \]
		where $ C = \big( \{0,1,2\} \times \{0,1,\ldots,2k+2\} \big) \setminus \{ (1,2k+2),(2,2k+2)\} $.
	\end{theorem}
	
	Let us observe that
	\begin{itemize}
		\item $(6k+9)+b(6k+13)=(b+1)(6k+13)-4$, for all $b\in\{0,1,\ldots,2k+1\}$;
		\item $2(6k+9)+b(6k+13)=(b+2)(6k+13)-8$, for all $b\in\{0,1,\ldots,2k+1\}$.
	\end{itemize}
	
	\begin{corollary}\label{cor12-3}
		If $k\in\mathbb{N}$ and $S=\langle 6k+7,6k+9,6k+13 \rangle$, then
		\[ \mathrm{Ap}(S,6k+7)=\{0; (6k+13)-4,6k+13; 2(6k+13)-8,2(6k+13)-4,2(6k+13); \]
		\[ \ldots; (2k+2)(6k+13)-8,(2k+2)(6k+13)-4,(2k+2)(6k+13); \]
		\[ (2k+3)(6k+13)-8 \}. \]
	\end{corollary}
	
	\begin{corollary}\label{cor13-3}
		If $k\in\mathbb{N}$ and $S=\langle 6k+7,6k+9,6k+13 \rangle$, then
		\begin{enumerate}
			\item $\mathrm{PF}(S) = \{ 12k^2+32k+19,12k^2+38k+24 \}$;
			\item $\mathrm{F}(S) = 12k^2+38k+24$;
			\item $\mathrm{g}(S) = 6k^2+20k+14$.
		\end{enumerate}
	\end{corollary}
	
	\begin{remark}\label{casetri1-2}
		If $p=6k+7$ with $k\in\mathbb{N}$, then $\mathrm{F}(p,p+2,p+6)=\frac{p^2+5p-12}{3}$.
	\end{remark}

	\subsubsection{Other triplets (4): $S=\langle 6k+5,6k+9,6k+11 \rangle$}
	
	Since $\gcd({6k+9,6k+11})=1$, we can assert that $S=\langle 6k+5,6k+9,6k+11 \rangle$ is a numerical semigroup.
	
	\begin{lemma}\label{lem11-4}
		If $k\in\mathbb{N}$, then we have the equalities
		\begin{enumerate}
			\item $3(6k+9)=1(6k+5)+2(6k+11)$;
			\item $(2k+3)(6k+11)=(2k+3)(6k+5)+2(6k+9)$;
			\item $1(6k+9)+(2k+1)(6k+11)=(2k+4)(6k+5)$.
		\end{enumerate}
	\end{lemma}
	
	\begin{theorem}\label{thm12-4}
		If $k\in\mathbb{N}$ and $S=\langle 6k+5,6k+9,6k+11 \rangle$, then
		\[ \mathrm{Ap}(S,6k+5)=\left\{ a(6k+9)+b(6k+11) \mid (a,b)\in C \right\}, \]
		where $ C = \big( \{0,1,2\} \times \{0,1,\ldots,2k\} \big) \cup \{ (0,2k+1),(0,2k+2)\} $.
	\end{theorem}
	
	Let us observe that
	\begin{itemize}
		\item $(6k+9)+b(6k+11)=(b+1)(6k+11)-2$, for all $b\in\{0,1,\ldots,2k+1\}$;
		\item $2(6k+9)+b(6k+11)=(b+2)(6k+11)-4$, for all $b\in\{0,1,\ldots,2k+1\}$.
	\end{itemize}
	
	\begin{corollary}\label{cor12-4}
		If $k\in\mathbb{N}$ and $S=\langle 6k+5,6k+9,6k+11 \rangle$, then
		\[ \mathrm{Ap}(S,6k+7)=\{0; (6k+11)-2,6k+11; 2(6k+11)-4,2(6k+11)-2,2(6k+11); \]
		\[ \ldots; (2k+1)(6k+11)-4,(2k+1)(6k+11)-2,(2k+1)(6k+11); \]
		\[ (2k+2)(6k+11)-4,(2k+2)(6k+11) \}. \]
	\end{corollary}
	
	\begin{corollary}\label{cor13-4}
		If $k\in\mathbb{N}$ and $S=\langle 6k+5,6k+9,6k+11 \rangle$, then
		\begin{enumerate}
			\item $\mathrm{PF}(S) = \{ 12k^2+28k+13,12k^2+28k+17 \}$;
			\item $\mathrm{F}(S) = 12k^2+28k+17$;
			\item $\mathrm{g}(S) = 6k^2+16k+10$.
		\end{enumerate}
	\end{corollary}
	
	\begin{remark}\label{casetri-2-2}
		If $p=6k+5$ with $k\in\mathbb{N}$, then $\mathrm{F}(p,p+4,p+6)=\frac{p^2+4p+6}{3}$.
	\end{remark}

	\subsection{Quintuplets}\label{quintuplets}
	
	Let us recall that a prime quintuplet is of the form $(p,p+2,p+6,p+8,p+12)$ or of the form $(p,p+4,p+6,p+10,p+12)$. This fact is improved in the following proposition.
	
	\begin{proposition}\label{prop31}
		We have that:
		\begin{enumerate}
			\item If $(p,p+2,p+6,p+8,p+12)$ is a prime quintuplet, then $p=6k+5$, with $k\in\mathbb{N}$.
			\item If $(p,p+4,p+6,p+10,p+12)$ is a prime quintuplet, then $p=6k+7$, with $k\in\mathbb{N}$.
		\end{enumerate}
	\end{proposition}
	
	As a consequence of the previous proposition, we have that a prime quintuplet is of one of the following two forms.
	\begin{enumerate}
		\item $(6k+5,6k+7,6k+11,6k+13,6k+17)$, with $k\in\mathbb{N}$.
		\item $(6k+7,6k+11,6k+13,6k+17,6k+19)$, with $k\in\mathbb{N}$.
	\end{enumerate}
	
	Since $\gcd({6k+5,6k+7})=\gcd({6k+11,6k+13})=1$, we can define two families of numerical semigroups ($\mathcal{I}_1$ and $\mathcal{I}_2$) associated with the prime quintuplets.
	\begin{itemize}
		\item $S\in\mathcal{I}_1$ if $S=\langle 6k+5,6k+7,6k+11,6k+13,6k+17 \rangle$, with $k\in\mathbb{N}\setminus\{0\}$.
		\item $S\in\mathcal{I}_2$ if $S=\langle 6k+7,6k+11,6k+13,6k+17,6k+19 \rangle$, with $k\in\mathbb{N}$.
	\end{itemize}
	Let us observe that $\langle 7,11,13,17,19 \rangle$ is a numerical semigroup with embedding dimension equal to five. Moreover, since $6k+17<2(6k+5)$ and $6k+19<2(6k+7)$ for all $k\geq 1$, we can assert that $e(S)=5$ for every $S\in\mathcal{I}_1\cup\mathcal{I}_2$.
	
	\begin{remark}\label{rem32}
		We have $S=\langle 1,5,7,11,13 \rangle = \langle 1 \rangle$ and $S=\langle 5,7,11,13,17 \rangle =\langle 5,7,11,13 \rangle$. Thus, $\mathrm{e}(S)<4$ in both cases. This is a reason for eliminating  $(1,5,7,11,13)$ and $(5,7,11,13,17)$ as possibilities in the families $\mathcal{I}_1$ and $\mathcal{I}_2$.  
	\end{remark}
	
	\begin{remark}\label{rem33}
		It is possible to improve Proposition~\ref{prop31} a little more. In fact, we have that $p=30k+11$, for some $k\in\mathbb{N}$, when $(p,p+2,p+6,p+8,p+12)$ and $p=30k+7$, for some $k\in\mathbb{N}$, when $(p,p+4,p+6,p+10,p+12)$. In any case, we obtain essentially the same conclusions for the families $\mathcal{I}_1$ and $\mathcal{I}_2$ associated with these expressions.
	\end{remark}

	\subsubsection{First case (family $\mathcal{I}_1$)}\label{case-quione}
	
	Numerical semigroups of the form $S=\langle 6k+5,6k+7,6k+11,6k+13,6k+17 \rangle$, where $k\in\mathbb{N}\setminus\{0\}$.
	
	\begin{lemma}\label{lem34}
		If $k\in\mathbb{N}$, then we have the equalities
		\begin{enumerate}
			\item $3(6k+7)=2(6k+5)+1(6k+11)$;
			\item $2(6k+11)=1(6k+5)+1(6k+17)$;
			\item $3(6k+13)=1(6k+5)+2(6k+17)$;
			\item $(k+1)(6k+17)=(k+2)(6k+5)+1(6k+7)$;
			\item $1(6k+7)+1(6k+11)=1(6k+5)+1(6k+13)$;
			\item $1(6k+11)+1(6k+13)=1(6k+7)+1(6k+17)$;
			\item $2(6k+7)+1(6k+13)=1(6k+5)+2(6k+11)=2(6k+5)+1(6k+17)$;
			\item $1(6k+7)+2(6k+13)=1(6k+5)+1(6k+11)+1(6k+17)$;
			\item $2(6k+7)+1(6k+17)=1(6k+5)+2(6k+13)$;
			\item $2(6k+13)+k(6k+17)=(k+3)(6k+5)+1(6k+11)$;
			\item $1(6k+7)+1(6k+13)+k(6k+17)=(k+4)(6k+5)$.
		\end{enumerate}
	\end{lemma}
	
	Now, let us take $ B = \{0,1\} \times \{0,1\} \times \{0,1,2\} \times \{0,1,\ldots,k-1\} $. Moreover, let us consider the usual product order on $\mathbb{N}^n$.
	
	\begin{theorem}\label{thm35}
		If $k\in\mathbb{N}\setminus\{0\}$ and $S=\langle 6k+5,6k+7,6k+11,6k+13,6k+17 \rangle$, then $\mathrm{Ap}(S,6k+5)$ is given by
		\[ \left\{ a(6k+7)+b(6k+11)+c(6k+13)+d(6k+17) \mid (a,b,c,d)\in C \right\}, \]
		where $ C = \{ (a,b,c,d) \in B \mid ab=0, \; bc=0, \; ac\not=2 \} \cup \{(2,0,0,0)\} \cup \{ (0,0,0,k), (1,0,0,k), (0,1,0,k), (0,0,1,k) \} $.
	\end{theorem}
	
	We can describe the Apéry set arranging its elements in increasing order.
	
	\begin{corollary}\label{cor36a}
		If $k\in\mathbb{N}\setminus\{0\}$ and $S=\langle 6k+5,6k+7,6k+11,6k+13,6k+17 \rangle$, then
		\[ \mathrm{Ap}(S,6k+5)=\{0; (6k+17)-10,(6k+17)-6, (6k+17)-4, 6k+17; 2(6k+7); \]
		\[ 2(6k+17)-14,2(6k+17)-10,2(6k+17)-8,2(6k+17)-6,2(6k+17)-4,2(6k+17); \]
		\[ \ldots; k(6k+17)-14,k(6k+17)-10,k(6k+17)-8, \]
		\[ k(6k+17)-6,k(6k+17)-4,k(6k+17); \]
		\[ (k+1)(6k+17)-14,(k+1)(6k+17)-10,(k+1)(6k+17)-8, \]
		\[ (k+1)(6k+17)-6,(k+1)(6k+17)-4 \}. \]
	\end{corollary}
	
	We identify a certain pattern: ``;'' separates certain groups of numbers.
	
	\begin{corollary}\label{cor36b}
		If $k\in\mathbb{N}\setminus\{0\}$ and $S=\langle 6k+5,6k+7,6k+11,6k+13,6k+17 \rangle$, then
		\begin{enumerate}
			\item $\mathrm{PF}(S) = \{ 6k+9,6k^2+17k-2,6k^2+17k+2,6k^2+17k+4,6k^2+17k+6,6k^2+17k+8 \}$;
			\item $\mathrm{F}(S) = 6k^2+17k+8$;
			\item $\mathrm{g}(S) = 3k^2+11k+7$.
		\end{enumerate}
	\end{corollary}
	
	\begin{remark}\label{casequi-1}
		By Corollary~\ref{cor36b}, if $p=6k+5$ for some $k\in\mathbb{N}\setminus\{0\}$ (in particular, if $(p,p+2,p+6,p+8,p+12)$ is a prime quintuplet with $p\geq11$), then $\mathrm{F}(p,p+2,p+6,p+8,p+12)=\frac{p^2+7p-12}{6}$.
	\end{remark}

	\subsubsection{Second case (family $\mathcal{I}_2$)}\label{case-quitwo}
	
	Numerical semigroups of the form $S=\langle 6k+7,6k+11,6k+13,6k+17,6k+19 \rangle$, where $k\in\mathbb{N}$.
	
	\begin{lemma}\label{lem37}
		If $k\in\mathbb{N}$, then we have the equalities
		\begin{enumerate}
			\item $3(6k+11)=1(6k+7)+2(6k+13)$;
			\item $2(6k+13)=1(6k+7)+1(6k+19)$;
			\item $3(6k+17)=1(6k+13)+2(6k+19)$;
			\item $(k+2)(6k+19)=(k+3)(6k+7)+1(6k+17)$;
			\item $1(6k+11)+1(6k+13)=1(6k+7)+1(6k+17)$;
			\item $1(6k+13)+1(6k+17)=1(6k+11)+1(6k+19)$;
			\item $2(6k+11)+1(6k+17)=1(6k+7)+1(6k+13)+1(6k+19)$;
			\item $1(6k+11)+2(6k+17)=1(6k+7)+2(6k+19)=2(6k+13)+1(6k+19)$;
			\item $2(6k+11)+1(6k+19)=1(6k+7)+2(6k+17)$;
			\item $1(6k+13)+(k+1)(6k+19)=(k+3)(6k+7)+1(6k+11)$;
			\item $1(6k+17)+(k+1)(6k+19)=(k+2)(6k+7)+2(6k+11)$;
			\item $2(6k+17)+k(6k+19)=(k+3)(6k+7)+1(6k+13)$.
		\end{enumerate}
	\end{lemma}
	
	Now, let us take $ B = \{0,1\} \times \{0,1\} \times \{0,1,2\} \times \{0,1,\ldots,k-1\} $. Moreover, let us consider the usual product order on $\mathbb{N}^n$.
	
	\begin{theorem}\label{thm38}
		If $k\in\mathbb{N}$ and $S=\langle 6k+7,6k+11,6k+13,6k+17,6k+19 \rangle$, then $\mathrm{Ap}(S,6k+7)$ is given by
		\[ \left\{ a(6k+11)+b(6k+13)+c(6k+17)+d(6k+19) \mid (a,b,c,d)\in C \right\}, \]
		where $ C = \{ (a,b,c,d) \in B \mid ab=0, \; bc=0, \; ac\not=2 \} \cup \{(2,0,0,0)\} \cup \{ (0,0,0,k), (1,0,0,k), (0,1,0,k), (0,0,1,k) \} \cup \{ (0,0,0,k+1), (1,0,0,k+1) \} $.
	\end{theorem}
	
	We can describe the Apéry set arranging its elements in increasing order.
	
	\begin{corollary}\label{cor39a}
		If $k\in\mathbb{N}$ and $S=\langle 6k+7,6k+11,6k+13,6k+17,6k+19 \rangle$, then
		\[ \mathrm{Ap}(S,6k+7)=\{0; (6k+19)-8,(6k+19)-6, (6k+19)-2, 6k+19; 2(6k+11); \]
		\[ 2(6k+19)-10,2(6k+19)-8,2(6k+19)-6,2(6k+19)-4,2(6k+19)-2,2(6k+19); \]
		\[ \ldots; (k+1)(6k+19)-10,(k+1)(6k+19)-8,(k+1)(6k+19)-6, \]
		\[ (k+1)(6k+19)-4,(k+1)(6k+19)-2,(k+1)(6k+19); (k+2)(6k+19)-8 \}. \]
	\end{corollary}
	
	We identify a certain pattern: ``;'' separates certain groups of numbers.
	
	\begin{corollary}\label{cor39b}
		If $k\in\mathbb{N}$ and $S=\langle 6k+7,6k+11,6k+13,6k+17,6k+19 \rangle$, then
		\begin{enumerate}
			\item $\mathrm{PF}(S) = \{ 15,23 \}$ if $k=0$; $\mathrm{PF}(S) = \{ 6k+15,6k^2+19k+2,6k^2+19k+8,6k^2+25k+23 \}$ if $k\geq 1$;
			\item $\mathrm{F}(S) = 6k^2+25k+23$;
			\item $\mathrm{g}(S) = 3k^2+19k+13$.
		\end{enumerate}
	\end{corollary}
	
	\begin{remark}\label{casequi-2}
		By Corollary~\ref{cor39b}, if $p=6k+7$ for some $k\in\mathbb{N}\setminus\{0\}$ (in particular, if $(p,p+4,p+6,p+10,p+12)$ is a prime quintuplet), then $\mathrm{F}(p,p+4,p+6,p+10,p+12)=\frac{p^2+11p+12}{6}$.
	\end{remark}

	\subsubsection{Other quintuplets (1): $S=\langle 6k+7,6k+9,6k+13,6k+15,6k+19 \rangle$}
	
	Since $\gcd({6k+7,6k+9})=1$, we can assert that $S=\langle 6k+7,6k+9,6k+13,6k+15,6k+19 \rangle$ is a numerical semigroup.
	
	\begin{lemma}\label{lem34-1}
		If $k\in\mathbb{N}$, then we have the equalities
		\begin{enumerate}
			\item $3(6k+9)=2(6k+7)+1(6k+13)$;
			\item $2(6k+13)=1(6k+7)+1(6k+19)$;
			\item $3(6k+15)=1(6k+7)+2(6k+19)$;
			\item $(k+2)(6k+19)=(k+2)(6k+7)+1(6k+9)+1(6k+15)$;
			\item $1(6k+9)+1(6k+13)=1(6k+7)+1(6k+15)$;
			\item $1(6k+13)+1(6k+15)=1(6k+9)+1(6k+19)$;
			\item $2(6k+9)+1(6k+15)=1(6k+7)+2(6k+13)=2(6k+7)+1(6k+19)$;
			\item $1(6k+9)+2(6k+15)=1(6k+7)+1(6k+13)+1(6k+19)$;
			\item $2(6k+9)+1(6k+19)=1(6k+7)+2(6k+15)$;
			\item $1(6k+9)+(k+1)(6k+19)=(k+4)(6k+7)$;
			\item $1(6k+13)+(k+1)(6k+19)=(k+2)(6k+7)+2(6k+9)$;
			\item $1(6k+15)+(k+1)(6k+19)=(k+1)(6k+7)+3(6k+9)$;
			\item $2(6k+15)+k(6k+19)=(k+3)(6k+7)+1(6k+9)$.
		\end{enumerate}
	\end{lemma}
	
	Now, let us take $ B = \{0,1\} \times \{0,1\} \times \{0,1,2\} \times \{0,1,\ldots,k\} $. Moreover, let us consider the usual product order on $\mathbb{N}^n$.
	
	\begin{theorem}\label{thm35-1}
		If $k\in\mathbb{N}$ and $S=\langle 6k+7,6k+9,6k+13,6k+15,6k+19 \rangle$, then $\mathrm{Ap}(S,6k+7)$ is given by
		\[ \left\{ a(6k+9)+b(6k+13)+c(6k+15)+d(6k+19) \mid (a,b,c,d)\in C \right\}, \]
		where $ C = \{ (a,b,c,d) \in B \mid ab=0, \; bc=0, \; ac\not=2, (c,d)\not=(2,k) \} \cup \{(2,0,0,0)\} \cup \{(0,0,0,k+1) \}$.
	\end{theorem}
	
	We can describe the Apéry set arranging its elements in increasing order.
	
	\begin{corollary}\label{cor36a-1}
		If $k\in\mathbb{N}$ and $S=\langle 6k+7,6k+9,6k+13,6k+15,6k+19 \rangle$, then
		\[ \mathrm{Ap}(S,6k+7)=\{0; (6k+19)-10,(6k+19)-6, (6k+19)-4, 6k+19; 2(6k+9); \]
		\[ 2(6k+19)-14,2(6k+19)-10,2(6k+19)-8,2(6k+19)-6,2(6k+19)-4,2(6k+19); \]
		\[ \ldots; (k+1)(6k+19)-14,(k+1)(6k+19)-10,(k+1)(6k+19)-8, \]
		\[ (k+1)(6k+19)-6,(k+1)(6k+19)-4,(k+1)(6k+19);(k+2)(6k+19)-14 \}. \]
	\end{corollary}
	
	We identify a certain pattern: ``;'' separates certain groups of numbers.
	
	\begin{corollary}\label{cor36b-1}
		If $k\in\mathbb{N}$ and $S=\langle 6k+7,6k+9,6k+13,6k+15,6k+19 \rangle$, then
		\begin{enumerate}
			\item $\mathrm{PF}(S) = \{6,11,12,17\}$ if $k=0$; $\mathrm{PF}(S) = \{ 6k+11,6k^2+19k+6,6k^2+19k+12,6k^2+25k+17 \}$ if $k\geq 1$;
			\item $\mathrm{F}(S) = 6k^2+25k+17$;
			\item $\mathrm{g}(S) = 3k^2+13k+11$.
		\end{enumerate}
	\end{corollary}
	
	\begin{remark}\label{casequi-1-1}
		From Corollary~\ref{cor36b-1}, if $p=6k+7$ with $k\in\mathbb{N}$, then $\mathrm{F}(p,p+2,p+6,p+8,p+12)=\frac{p^2+11p-24}{6}$.
	\end{remark}

	\subsubsection{Other quintuplets (2): $S=\langle 6k+9,6k+11,6k+15,6k+17,6k+21 \rangle$}
	
	Since $\gcd({6k+9,6k+11})=1$, we can assert that $S=\langle 6k+9,6k+11,6k+15,6k+17,6k+21 \rangle$ is a numerical semigroup.
	
	\begin{lemma}\label{lem34-2}
		If $k\in\mathbb{N}$, then we have the equalities
		\begin{enumerate}
			\item $3(6k+11)=2(6k+9)+1(6k+15)$;
			\item $2(6k+15)=1(6k+9)+1(6k+21)$;
			\item $3(6k+17)=1(6k+9)+2(6k+21)$;
			\item $(k+2)(6k+21)=(k+1)(6k+9)+3(6k+11)$;
			\item $1(6k+11)+1(6k+15)=1(6k+9)+1(6k+17)$;
			\item $1(6k+15)+1(6k+17)=1(6k+11)+1(6k+21)$;
			\item $2(6k+11)+1(6k+17)=1(6k+9)+2(6k+15)=2(6k+9)+1(6k+21)$;
			\item $1(6k+11)+2(6k+17)=1(6k+9)+1(6k+15)+1(6k+21)$;
			\item $2(6k+11)+1(6k+21)=1(6k+9)+2(6k+17)$;
			\item $1(6k+15)+(k+1)(6k+21)=(k+4)(6k+9)$;
			\item $2(6k+17)+(k+1)(6k+21)=(k+3)(6k+9)+1(6k+11)$.
		\end{enumerate}
	\end{lemma}
	
	Now, let us take $ B = \{0,1\} \times \{0,1\} \times \{0,1,2\} \times \{0,1,\ldots,k\} $. Moreover, let us consider the usual product order on $\mathbb{N}^n$.
	
	\begin{theorem}\label{thm35-2}
		If $k\in\mathbb{N}$ and $S=\langle 6k+9,6k+11,6k+15,6k+17,6k+21 \rangle$, then $\mathrm{Ap}(S,6k+9)$ is given by
		\[ \left\{ a(6k+11)+b(6k+15)+c(6k+17)+d(6k+21) \mid (a,b,c,d)\in C \right\}, \]
		where $ C = \{ (a,b,c,d) \in B \mid ab=0, \; bc=0, \; ac\not=2 \} \cup \{(2,0,0,0)\} \cup \{(0,0,0,k+1), (1,0,0,k+1)\} $.
	\end{theorem}
	
	We can describe the Apéry set arranging its elements in increasing order.
	
	\begin{corollary}\label{cor36a-2}
		If $k\in\mathbb{N}$ and $S=\langle 6k+9,6k+11,6k+15,6k+17,6k+21 \rangle$, then
		\[ \mathrm{Ap}(S,6k+9)=\{0; (6k+21)-10,(6k+21)-6, (6k+21)-4, 6k+21; 2(6k+11); \]
		\[ 2(6k+21)-14,2(6k+21)-10,2(6k+21)-8,2(6k+21)-6,2(6k+21)-4,2(6k+21); \]
		\[ \ldots; (k+1)(6k+21)-14,(k+1)(6k+21)-10,(k+1)(6k+21)-8, \]
		\[ (k+1)(6k+21)-6,(k+1)(6k+21)-4,(k+1)(6k+21); \]
		\[ (k+2)(6k+21)-14,(k+2)(6k+21)-10,(k+2)(6k+21)-8 \}. \]
	\end{corollary}
	
	We identify a certain pattern: ``;'' separates certain groups of numbers.
	
	\begin{corollary}\label{cor36b-2}
		If $k\in\mathbb{N}$ and $S=\langle 6k+9,6k+11,6k+15,6k+17,6k+21 \rangle$, then
		\begin{enumerate}
			\item $\mathrm{PF}(S) = \{ 6k+13,6k^2+27k+19,6k^2+27k+23,6k^2+27k+25 \}$;
			\item $\mathrm{F}(S) = 6k^2+27k+25$;
			\item $\mathrm{g}(S) = 3k^2+15k+16$.
		\end{enumerate}
	\end{corollary}
	
	\begin{remark}\label{casequi-1-2}
		From Corollary~\ref{cor36b-2}, if $p=6k+9$ with $k\in\mathbb{N}$, then $\mathrm{F}(p,p+2,p+6,p+8,p+12)=\frac{p^2+9p-12}{6}$.
	\end{remark}

	\subsubsection{Other quintuplets (3): $S=\langle 6k+5,6k+9,6k+11,6k+15,6k+17 \rangle$}
	
	Since $\gcd({6k+9,6k+11})=1$, we can assert that $S=\langle 6k+5,6k+9,6k+11,6k+15,6k+17 \rangle$ is a numerical semigroup.
	
	\begin{lemma}\label{lem37-1}
		If $k\in\mathbb{N}$, then we have the equalities
		\begin{enumerate}
			\item $3(6k+9)=1(6k+5)+2(6k+11)$;
			\item $2(6k+11)=1(6k+5)+1(6k+17)$;
			\item $3(6k+15)=1(6k+11)+2(6k+17)$;
			\item $(k+2)(6k+17)=(k+1)(6k+5)+2(6k+9)+1(6k+11)$;
			\item $1(6k+9)+1(6k+11)=1(6k+5)+1(6k+15)$;
			\item $1(6k+11)+1(6k+15)=1(6k+9)+1(6k+17)$;
			\item $2(6k+9)+1(6k+15)=1(6k+5)+1(6k+11)+1(6k+17)$;
			\item $1(6k+9)+2(6k+15)=1(6k+5)+2(6k+17)=2(6k+11)+1(6k+17)$;
			\item $2(6k+9)+1(6k+17)=1(6k+5)+2(6k+15)$;
			\item $1(6k+9)+(k+1)(6k+17)=(k+3)(6k+5)+1(6k+11)$;
			\item $1(6k+11)+(k+1)(6k+17)=(k+2)(6k+5)+2(6k+9)$;
			\item $1(6k+15)+k(6k+17)=(k+3)(6k+5)$.
		\end{enumerate}
	\end{lemma}
	
	Now, let us take $ B = \{0,1\} \times \{0,1\} \times \{0,1,2\} \times \{0,1,\ldots,k-1\} $. Moreover, let us consider the usual product order on $\mathbb{N}^n$.
	
	\begin{theorem}\label{thm38-1}
		If $k\in\mathbb{N}\setminus\{0\}$ and $S=\langle 6k+5,6k+9,6k+11,6k+15,6k+17 \rangle$, then $\mathrm{Ap}(S,6k+5)$ is given by
		\[ \left\{ a(6k+9)+b(6k+11)+c(6k+15)+d(6k+17) \mid (a,b,c,d)\in C \right\}, \]
		where $ C = \{ (a,b,c,d) \in B \mid ab=0, \; bc=0, \; ac\not=2 \} \cup \{(2,0,0,0)\} \cup \{(0,0,0,k), (1,0,0,k), (0,1,0,k)\} \cup \{(0,0,0,k+1)\} $.
	\end{theorem}
	
	We can describe the Apéry set arranging its elements in increasing order.
	
	\begin{corollary}\label{cor39a-1}
		If $k\in\mathbb{N}\setminus\{0\}$ and $S=\langle 6k+5,6k+9,6k+11,6k+15,6k+17 \rangle$, then
		\[ \mathrm{Ap}(S,6k+5)=\{0; (6k+17)-8,(6k+17)-6, (6k+17)-2, 6k+17; 2(6k+9); \]
		\[ 2(6k+17)-10,2(6k+17)-8,2(6k+17)-6,2(6k+17)-4,2(6k+17)-2,2(6k+17); \]
		\[ \ldots; k(6k+17)-10,k(6k+17)-8,k(6k+17)-6,k(6k+17)-4, \]
		\[ k(6k+17)-2,k(6k+17); (k+1)(6k+17)-10,(k+1)(6k+17)-8, \]
		\[ (k+1)(6k+17)-6,(k+1)(6k+17)-4,(k+1)(6k+17) \}. \]
	\end{corollary}
	
	We identify a certain pattern: ``;'' separates certain groups of numbers.
	
	\begin{corollary}\label{cor39b-1}
		If $k\in\mathbb{N}\setminus\{0\}$ and $S=\langle 6k+5,6k+9,6k+11,6k+15,6k+17 \rangle$, then
		\begin{enumerate}
			\item $\mathrm{PF}(S) = \{ 6k+13,6k^2+17k+2,6k^2+17k+4,6k^2+17k+6,6k^2+17k+8,6k^2+17k+12 \}$ if $k\geq 1$;
			\item $\mathrm{F}(S) = 6k^2+17k+12$;
			\item $\mathrm{g}(S) = 3k^2+11k+9$.
		\end{enumerate}
	\end{corollary}
	
	\begin{remark}\label{casequi-2-1}
		From Corollary~\ref{cor39b-1}, if $p=6k+5$ with $k\in\mathbb{N}\setminus\{0\}$, then $\mathrm{F}(p,p+4,p+6,p+10,p+12)=\frac{p^2+7p+12}{6}$.
	\end{remark}

	\subsubsection{Other quintuplets (4): $S=\langle 6k+9,6k+13,6k+15,6k+19,6k+21 \rangle$}
	
	Since $\gcd({6k+13,6k+15})=1$, we can assert that $S=\langle 6k+9,6k+13,6k+15,6k+19,6k+21 \rangle$ is a numerical semigroup.
	
	\begin{lemma}\label{lem37-2}
		If $k\in\mathbb{N}$, then we have the equalities
		\begin{enumerate}
			\item $3(6k+13)=1(6k+9)+2(6k+15)$;
			\item $2(6k+15)=1(6k+9)+1(6k+21)$;
			\item $3(6k+19)=1(6k+15)+2(6k+21)$;
			\item $(k+2)(6k+21)=(k+3)(6k+9)+1(6k+15)$;
			\item $1(6k+13)+1(6k+15)=1(6k+9)+1(6k+19)$;
			\item $1(6k+15)+1(6k+19)=1(6k+13)+1(6k+21)$;
			\item $2(6k+13)+1(6k+19)=1(6k+9)+1(6k+15)+1(6k+21)$;
			\item $1(6k+13)+2(6k+19)=1(6k+9)+2(6k+21)=2(6k+15)+1(6k+21)$;
			\item $2(6k+13)+1(6k+21)=1(6k+9)+2(6k+19)$;
			\item $1(6k+15)+(k+1)(6k+21)=(k+4)(6k+9)$;
			\item $1(6k+19)+(k+1)(6k+21)=(k+3)(6k+9)+1(6k+13)$.
		\end{enumerate}
	\end{lemma}
	
	Now, let us take $ B = \{0,1\} \times \{0,1\} \times \{0,1,2\} \times \{0,1,\ldots,k\} $. Moreover, let us consider the usual product order on $\mathbb{N}^n$.
	
	\begin{theorem}\label{thm38-2}
		If $k\in\mathbb{N}$ and $S=\langle 6k+9,6k+13,6k+15,6k+19,6k+21 \rangle$, then $\mathrm{Ap}(S,6k+9)$ is given by
		\[ \left\{ a(6k+13)+b(6k+15)+c(6k+19)+d(6k+21) \mid (a,b,c,d)\in C \right\}, \]
		where $ C = \{ (a,b,c,d) \in B \mid ab=0, \; bc=0, \; ac\not=2 \} \cup \{(2,0,0,0)\} \cup \{(0,0,0,k+1), (1,0,0,k+1)\} $.
	\end{theorem}
	
	We can describe the Apéry set arranging its elements in increasing order.
	
	\begin{corollary}\label{cor39a-2}
		If $k\in\mathbb{N}$ and $S=\langle 6k+9,6k+13,6k+15,6k+19,6k+21 \rangle$, then
		\[ \mathrm{Ap}(S,6k+9)=\{0; (6k+21)-8,(6k+21)-6,(6k+21)-2,6k+21; 2(6k+13); \]
		\[ 2(6k+21)-10,2(6k+21)-8,2(6k+21)-6,2(6k+21)-4,2(6k+21)-2,2(6k+21); \]
		\[ \ldots; (k+1)(6k+21)-10,(k+1)(6k+21)-8,(k+1)(6k+21)-6, \]
		\[ (k+1)(6k+21)-4,(k+1)(6k+21)-2,(k+1)(6k+21); \]
		\[ (k+2)(6k+21)-10,(k+2)(6k+21)-8,(k+2)(6k+21)-4 \}. \]
	\end{corollary}
	
	We identify a certain pattern: ``;'' separates certain groups of numbers.
	
	\begin{corollary}\label{cor39b-2}
		If $k\in\mathbb{N}$ and $S=\langle 6k+9,6k+13,6k+15,6k+19,6k+21 \rangle$, then
		\begin{enumerate}
			\item $\mathrm{PF}(S) = \{ 6k+17,6k^2+27k+23,6k^2+27k+25,6k^2+27k+29 \}$;
			\item $\mathrm{F}(S) = 6k^2+27k+29$;
			\item $\mathrm{g}(S) = 3k^2+15k+18$.
		\end{enumerate}
	\end{corollary}
	
	\begin{remark}\label{casequi-2-2}
		From Corollary~\ref{cor39b-2}, if $p=6k+9$ with $k\in\mathbb{N}$, then $\mathrm{F}(p,p+4,p+6,p+10,p+12)=\frac{p^2+9p+12}{6}$.
	\end{remark}

	\subsection{Sextuplets}\label{sextuplets}
	
	Let us recall that a prime triplet is of the form $(p,p+4,p+6,p+10,p+12,p+16)$. This fact is improved in the following proposition.
	
	\begin{proposition}\label{prop41}
		Except $\{7, 11, 13, 17, 19, 23\}$, all prime sextuplets are of the form $\{210n + 97, 210n + 101, 210n + 103, 210n + 107, 210n + 109, 210n + 113\}$ with $n\in\mathbb{N}$.
	\end{proposition}
	
	However, to study the Frobenius problem, we are going to consider the following four cases.
	\begin{enumerate}
		\item $p=8k+7$, with $k\in\mathbb{N}$;
		\item $p=8k+9$, with $k\in\mathbb{N}$;
		\item $p=8k+11$, with $k\in\mathbb{N}$;
		\item $p=8k+13$, with $k\in\mathbb{N}$.
	\end{enumerate}
	
	Since $\gcd({13,15})=\gcd({15,17})=\gcd({17,19})=1$, we can define four families of numerical semigroups ($\mathcal{S}_1$, $\mathcal{S}_2$, $\mathcal{S}_3$, and $\mathcal{S}_4$) associated with the prime sextuplets.
	\begin{itemize}
		\item $S\in\mathcal{S}_1$ if $S=\langle 8k+7,8k+11,8k+13,8k+17,8k+19,8k+23 \rangle$, with $k\in\mathbb{N}$.
		\item $S\in\mathcal{S}_2$ if $S=\langle 8k+9,8k+13,8k+15,8k+19,8k+21,8k+25 \rangle$, with $k\in\mathbb{N}$.
		\item $S\in\mathcal{S}_3$ if $S=\langle 8k+11,8k+15,8k+17,8k+21,8k+23,8k+27 \rangle$, with $k\in\mathbb{N}$.
		\item $S\in\mathcal{S}_4$ if $S=\langle 8k+13,8k+17,8k+19,8k+23,8k+25,8k+29 \rangle$, with $k\in\mathbb{N}$.
	\end{itemize}
	Let us note that $\langle 7,11,13,17,19,23 \rangle$, $\langle 9,13,15,19,21,25 \rangle$, $\langle 11,15,17,21,23,27 \rangle$, $\langle 13,17,19,23,25,29 \rangle$, and $\langle 15,19,21,25,27,31 \rangle$ are numerical semigroups with embedding dimension equal to six. Moreover, since $8k+23<2(8k+7)$ for all $k\geq 2$ and $8k+25<2(8k+9)$, $8k+27<2(8k+11)$, and $8k+29<2(8k+13)$ for all $k\geq 1$, we can assert that $e(S)=6$ for every $S\in\mathcal{S}_1\cup\mathcal{S}_2\cup\mathcal{S}_3\cup\mathcal{S}_4$.

	\subsubsection{First case (family $\mathcal{S}_1$)}\label{case-sexone}
	
	Numerical semigroups of the form $S=\langle 8k+7,8k+11,8k+13,8k+17,8k+19,8k+23 \rangle$, where $k\in\mathbb{N}$.
	
	\begin{lemma}\label{lem44}
		If $k\in\mathbb{N}$, then we have the equalities
		\begin{enumerate}
			\item $3(8k+11)=1(8k+7)+2(8k+13)$;
			\item $2(8k+13)=1(8k+7)+1(8k+19)$;
			\item $2(8k+17)=1(8k+11)+1(8k+23)$;
			\item $3(8k+19)=1(8k+11)+2(8k+23)$;
			\item $(k+2)(8k+23)=(k+2)(8k+7)+1(8k+13)+1(8k+19)$;
			\item $1(8k+11)+1(8k+13)=1(8k+7)+1(8k+17)$;
			\item $1(8k+11)+1(8k+19)=1(8k+7)+1(8k+23)$;
			\item $1(8k+13)+1(8k+17)=1(8k+7)+1(8k+23)$;
			\item $1(8k+17)+1(8k+19)=1(8k+13)+1(8k+23)$;
			\item $2(8k+11)+1(8k+17)=1(8k+7)+1(8k+13)+1(8k+19)$;
			\item $2(8k+11)+1(8k+23)=1(8k+7)+2(8k+19)$;
			\item $1(8k+13)+2(8k+19)=1(8k+11)+1(8k+17)+1(8k+23)$;
			\item $1(8k+11)+1(8k+17)+k(8k+23)=(k+4)(8k+7)$;
			\item $1(8k+13)+1(8k+19)+k(8k+23)=(k+3)(8k+7)+1(8k+11)$;
			\item $2(8k+19)+k(8k+23)=(k+3)(8k+7)+1(8k+17)$;
			\item $1(8k+11)+(k+1)(8k+23)=(k+3)(8k+7)+1(8k+13)$;
			\item $1(8k+13)+(k+1)(8k+23)=(k+2)(8k+7)+2(8k+11)$;
			\item $1(8k+17)+(k+1)(8k+23)=(k+1)(8k+7)+3(8k+11)$;
			\item $1(8k+19)+(k+1)(8k+23)=(k+1)(8k+7)+2(8k+11)+1(8k+13)$.
		\end{enumerate}
	\end{lemma}
	
	Now, let us take $ B = \{0,1\} \times \{0,1\} \times \{0,1\} \times \{0,1,2\} \times \{0,1,\ldots,k-1\} $. Moreover, let us consider the usual product order on $\mathbb{N}^n$.
	
	\begin{theorem}\label{thm45}
		If $k\in\mathbb{N}$ and $S=\langle 8k+7,8k+11,8k+13,8k+17,8k+19,8k+23 \rangle$, then $\mathrm{Ap}(S,8k+7)$ is given by
		\[ \left\{ a(8k+11)+b(8k+13)+c(8k+17)+d(8k+19)+e(8k+23) \mid (a,b,c,d,e)\in C \right\}, \]
		where $ C = \{ (a,b,c,d,e) \in B \mid ab=0, \; ad=0, \; bc=0, \; cd=0, \; bd\not=2 \} \cup \{(2,0,0,0,0)\} \cup \{(0,0,0,0,k),(1,0,0,0,k),(0,1,0,0,k),(0,0,1,0,k),(0,0,0,1,k)\}$ $\cup \{(0,0,0,0,k+1)\} $.
	\end{theorem}
	
	We can describe the Apéry set arranging its elements in increasing order.
	
	\begin{corollary}\label{cor46a}
		If $k\in\mathbb{N}\setminus\{0\}$ and $S=\langle 8k+7,8k+11,8k+13,8k+17,8k+19,8k+23 \rangle$, then
		\[ \mathrm{Ap}(S,8k+7)=\{0; (8k+23)-12,(8k+23)-10,(8k+23)-6,(8k+23)-4,8k+23; \]
		\[ 2(8k+11); 2(8k+23)-18,2(8k+23)-14,2(8k+23)-12,2(8k+23)-10, \]
		\[ 2(8k+23)-8,2(8k+23)-6,2(8k+23)-4,2(8k+23); \ldots; \]
		\[ (k+1)(8k+23)-18,(k+1)(8k+23)-14,(k+1)(8k+23)-12,(k+1)(8k+23)-10, \]
		\[ (k+1)(8k+23)-8,(k+1)(8k+23)-6,(k+1)(8k+23)-4,(k+1)(8k+23) \}. \]
		Moreover, if $S=\langle 7,11,13,17,19,23 \rangle$, then $\mathrm{Ap}(S,7)=\{0;11,13,17,19,23;22\}$.
	\end{corollary}
	
	We identify a certain pattern: ``;'' separates certain groups of numbers.
	
	\begin{corollary}\label{cor46ab}
		If $S=\langle 7,11,13,17,19,23 \rangle$, then
		\begin{enumerate}
			\item $\mathrm{PF}(S) = \{ 6,10,12,15,16 \}$;
			\item $\mathrm{F}(S) = 16$;
			\item $\mathrm{g}(S) = 12$.
		\end{enumerate}
	\end{corollary}
	
	\begin{corollary}\label{cor46b}
		If $k\in\mathbb{N}\setminus\{0\}$ and $S=\langle 8k+7,8k+11,8k+13,8k+17,8k+19,8k+23 \rangle$, then
		\begin{enumerate}
			\item $\mathrm{PF}(S) = \{ 8k+15,8k^2+23k-2,8k^2+23k+2,8k^2+23k+4,8k^2+23k+6,8k^2+23k+8,8k^2+23k+10,8k^2+23k+12,8k^2+23k+16 \}$;
			\item $\mathrm{F}(S) = 8k^2+23k+16$;
			\item $\mathrm{g}(S) = 4k^2+16k+12$.
		\end{enumerate}
	\end{corollary}
	
	\begin{remark}\label{casesex-1}
		From Corollaries~\ref{cor46ab} and \ref{cor46b}, we deduce that, if $p=8k+7$ with $k\in\mathbb{N}$, then $\mathrm{F}(p,p+4,p+6,p+10,p+12,p+16)=\frac{p^2+9p+16}{8}$.
	\end{remark}

	\subsubsection{Second case (family $\mathcal{S}_2$)}\label{case-sextwo}
	
	Numerical semigroups of the form $S=\langle 8k+9,8k+13,8k+15,8k+19,8k+21,8k+25 \rangle$, where $k\in\mathbb{N}$.
	
	\begin{lemma}\label{lem44-2}
		If $k\in\mathbb{N}$, then we have the equalities
		\begin{enumerate}
			\item $3(8k+13)=1(8k+9)+2(8k+15)$;
			\item $2(8k+15)=1(8k+9)+1(8k+21)$;
			\item $2(8k+19)=1(8k+13)+1(8k+25)$;
			\item $3(8k+21)=1(8k+13)+2(8k+25)$;
			\item $(k+2)(8k+25)=(k+2)(8k+9)+1(8k+13)+1(8k+19)$;
			\item $1(8k+13)+1(8k+15)=1(8k+9)+1(8k+19)$;
			\item $1(8k+13)+1(8k+21)=1(8k+9)+1(8k+25)$;
			\item $1(8k+15)+1(8k+19)=1(8k+9)+1(8k+25)$;
			\item $1(8k+19)+1(8k+21)=1(8k+15)+1(8k+25)$;
			\item $2(8k+13)+1(8k+19)=1(8k+9)+1(8k+15)+1(8k+21)$;
			\item $2(8k+13)+1(8k+25)=1(8k+9)+2(8k+21)$;
			\item $1(8k+15)+2(8k+21)=1(8k+13)+1(8k+19)+1(8k+25)$;
			\item $2(8k+21)+k(8k+25)=(k+3)(8k+9)+1(8k+15)$;
			\item $1(8k+15)+1(8k+21)+k(8k+25)=(k+4)(8k+9)$;
			\item $1(8k+15)+(k+1)(8k+25)=(k+3)(8k+9)+1(8k+13)$;
			\item $1(8k+19)+(k+1)(8k+25)=(k+2)(8k+9)+1(8k+15)+1(8k+21)$;
			\item $1(8k+21)+(k+1)(8k+25)=(k+2)(8k+9)+1(8k+13)+1(8k+15)$.
		\end{enumerate}
	\end{lemma}
	
	Now, let us take $ B = \{0,1\} \times \{0,1\} \times \{0,1\} \times \{0,1,2\} \times \{0,1,\ldots,k-1\} $. Moreover, let us consider the usual product order on $\mathbb{N}^n$.
	
	\begin{theorem}\label{thm45-2}
		If $k\in\mathbb{N}$ and $S=\langle 8k+9,8k+13,8k+15,8k+19,8k+21,8k+25 \rangle$, then $\mathrm{Ap}(S,8k+9)$ is given by
		\[ \left\{ a(8k+13)+b(8k+15)+c(8k+19)+d(8k+21)+e(8k+25) \mid (a,b,c,d,e)\in C \right\}, \]
		where $ C = \{ (a,b,c,d,e) \in B \mid ab=0, \; ad=0, \; bc=0, \; cd=0, \; bd\not=2 \}  \cup \{(2,0,0,0,0)\} \cup \{(0,0,0,0,k),(1,0,0,0,k),(0,1,0,0,k),(0,0,1,0,k),(1,0,1,0,k),$ $(0,0,0,1,k)\} \cup \{(0,0,0,0,k+1),(1,0,0,0,k+1)\}$.
	\end{theorem}
	
	We can describe the Apéry set arranging its elements in increasing order.
	
	\begin{corollary}\label{cor46a-2}
		If $k\in\mathbb{N}\setminus\{0\}$ and $S=\langle 8k+9,8k+13,8k+15,8k+19,8k+21,8k+25 \rangle$, then
		\[ \mathrm{Ap}(S,8k+9)=\{0; (8k+25)-12,(8k+25)-10,(8k+25)-6,(8k+25)-4,8k+25; \]
		\[ 2(8k+13); 2(8k+25)-18,2(8k+25)-14,2(8k+25)-12,2(8k+25)-10, \]
		\[ 2(8k+25)-8,2(8k+25)-6,2(8k+25)-4,2(8k+25); \ldots; \]
		\[ (k+1)(8k+25)-18,(k+1)(8k+25)-14,(k+1)(8k+25)-12,(k+1)(8k+25)-10, \]
		\[ (k+1)(8k+25)-8,(k+1)(8k+25)-6,(k+1)(8k+25)-4,(k+1)(8k+25); \]
		\[ (k+2)(8k+25)-18,(k+2)(8k+25)-12 \}. \]
		Moreover, if $S=\langle 9,13,15,19,21,25 \rangle$, then
		\[ \mathrm{Ap}(S,9)=\{0;13,15,19,21,25;26;32,38\}. \]
	\end{corollary}
	
	We identify a certain pattern: ``;'' separates certain groups of numbers.
	
	
	\begin{corollary}\label{cor46b-2}
		If $k\in\mathbb{N}$ and $S=\langle 8k+9,8k+13,8k+15,8k+19,8k+21,8k+25 \rangle$, then
		\begin{enumerate}
			\item $\mathrm{PF}(S) = \{ 8k+17,8k^2+25k+6,8k^2+25k+12,8k^2+33k+23,8k^2+33k+29 \}$;
			\item $\mathrm{F}(S) = 8k^2+33k+29$;
			\item $\mathrm{g}(S) = 4k^2+18k+17$.
		\end{enumerate}
	\end{corollary}
	
	\begin{remark}\label{casesex-2}
		From Corollary~\ref{cor46b-2}, we deduce that, if $p=8k+9$ with $k\in\mathbb{N}$, then $\mathrm{F}(p,p+4,p+6,p+10,p+12,p+16)=\frac{p^2+15p+16}{8}$.
	\end{remark}

	\subsubsection{Third case (family $\mathcal{S}_3$)}\label{case-sexthree}
	
	Numerical semigroups of the form $S=\langle 8k+11,8k+15,8k+17,8k+21,8k+23,8k+27 \rangle$, where $k\in\mathbb{N}$.
	
	\begin{lemma}\label{lem44-3}
		If $k\in\mathbb{N}$, then we have the equalities
		\begin{enumerate}
			\item $3(8k+15)=1(8k+11)+2(8k+17)$;
			\item $2(8k+17)=1(8k+11)+1(8k+23)$;
			\item $2(8k+21)=1(8k+15)+1(8k+27)$;
			\item $3(8k+23)=1(8k+15)+2(8k+27)$;
			\item $(k+2)(8k+27)=(k+2)(8k+11)+1(8k+15)+1(8k+17)$;
			\item $1(8k+15)+1(8k+17)=1(8k+11)+1(8k+21)$;
			\item $1(8k+15)+1(8k+23)=1(8k+11)+1(8k+27)$;
			\item $1(8k+17)+1(8k+21)=1(8k+11)+1(8k+27)$;
			\item $1(8k+21)+1(8k+23)=1(8k+17)+1(8k+27)$;
			\item $2(8k+15)+1(8k+21)=1(8k+11)+1(8k+17)+1(8k+23)$;
			\item $2(8k+15)+1(8k+27)=1(8k+11)+2(8k+23)$;
			\item $1(8k+17)+2(8k+23)=1(8k+15)+1(8k+21)+1(8k+27)$;
			\item $1(8k+17)+(k+1)(8k+27)=(k+4)(8k+11)$;
			\item $1(8k+21)+(k+1)(8k+27)=(k+3)(8k+11)+1(8k+15)$;
			\item $1(8k+23)+(k+1)(8k+27)=(k+3)(8k+11)+1(8k+17)$.
		\end{enumerate}
	\end{lemma}
	
	Now, let us take $ B = \{0,1\} \times \{0,1\} \times \{0,1\} \times \{0,1,2\} \times \{0,1,\ldots,k\} $. Moreover, let us consider the usual product order on $\mathbb{N}^n$.
	
	\begin{theorem}\label{thm45-3}
		If $k\in\mathbb{N}$ and $S=\langle 8k+11,8k+15,8k+17,8k+21,8k+23,8k+27 \rangle$, then $\mathrm{Ap}(S,8k+11)$ is given by
		\[ \left\{ a(8k+15)+b(8k+17)+c(8k+21)+d(8k+23)+e(8k+27) \mid (a,b,c,d,e)\in C \right\}, \]
		where $ C = \{ (a,b,c,d,e) \in B \mid ab=0, \; ad=0, \; bc=0, \; cd=0, \; bd\not=2 \}  \cup \{(2,0,0,0,0)\} \cup \{(0,0,0,0,k+1),(1,0,0,0,k+1)\}$.
	\end{theorem}
	
	We can describe the Apéry set arranging its elements in increasing order.
	
	\begin{corollary}\label{cor46a-3}
		If $k\in\mathbb{N}\setminus\{0\}$ and $S=\langle 8k+11,8k+15,8k+17,8k+21,8k+23,8k+27 \rangle$, then
		\[ \mathrm{Ap}(S,8k+11)=\{0; (8k+27)-12,(8k+27)-10,(8k+27)-6,(8k+27)-4,8k+27; \]
		\[ 2(8k+15); 2(8k+27)-18,2(8k+27)-14,2(8k+27)-12,2(8k+27)-10, \]
		\[ 2(8k+27)-8,2(8k+27)-6,2(8k+27)-4,2(8k+27); \ldots; \]
		\[ (k+1)(8k+27)-18,(k+1)(8k+27)-14,(k+1)(8k+27)-12,(k+1)(8k+27)-10, \]
		\[ (k+1)(8k+27)-8,(k+1)(8k+27)-6,(k+1)(8k+27)-4,(k+1)(8k+27); \]
		\[ (k+2)(8k+27)-18,(k+2)(8k+27)-14,(k+2)(8k+27)-12,(k+2)(8k+27)-8 \}. \]
		Moreover, if $S=\langle 11,15,17,21,23,27 \rangle$, then
		\[ \mathrm{Ap}(S,11)=\{0;15,17,21,23,27;30;36,40,42,46\}. \]
	\end{corollary}
	
	We identify a certain pattern: ``;'' separates certain groups of numbers.
	
	\begin{corollary}\label{cor46b-3}
		If $k\in\mathbb{N}$ and $S=\langle 8k+11,8k+15,8k+17,8k+21,8k+23,8k+27 \rangle$, then
		\begin{enumerate}
			\item $\mathrm{PF}(S) = \{ 8k+19,8k^2+35k+25,8k^2+35k+29,8k^2+35k+31,8k^2+35k+35 \}$;
			\item $\mathrm{F}(S) = 8k^2+35k+35$;
			\item $\mathrm{g}(S) = 4k^2+20k+22$.
		\end{enumerate}
	\end{corollary}
	
	\begin{remark}\label{casesex-3}
		From Corollary~\ref{cor46b-3}, we deduce that, if $p=8k+11$ with $k\in\mathbb{N}$, then $\mathrm{F}(p,p+4,p+6,p+10,p+12,p+16)=\frac{p^2+13p+16}{8}$.
	\end{remark}

	\subsubsection{Fourth case (family $\mathcal{S}_4$)}\label{case-sexfour}
	
	Numerical semigroups of the form $S=\langle 8k+13,8k+17,8k+19,8k+23,8k+25,8k+29 \rangle$, where $k\in\mathbb{N}$.
	
	\begin{lemma}\label{lem44-4}
		If $k\in\mathbb{N}$, then we have the equalities
		\begin{enumerate}
			\item $3(8k+17)=1(8k+13)+2(8k+19)$;
			\item $2(8k+19)=1(8k+13)+1(8k+25)$;
			\item $2(8k+23)=1(8k+17)+1(8k+29)$;
			\item $3(8k+25)=1(8k+17)+2(8k+29)$;
			\item $(k+2)(8k+29)=(k+3)(8k+13)+1(8k+19)$;
			\item $1(8k+17)+1(8k+19)=1(8k+13)+1(8k+23)$;
			\item $1(8k+17)+1(8k+25)=1(8k+13)+1(8k+29)$;
			\item $1(8k+19)+1(8k+23)=1(8k+13)+1(8k+29)$;
			\item $1(8k+23)+1(8k+25)=1(8k+19)+1(8k+29)$;
			\item $2(8k+17)+1(8k+23)=1(8k+13)+1(8k+19)+1(8k+25)$;
			\item $2(8k+17)+1(8k+29)=1(8k+13)+2(8k+25)$;
			\item $1(8k+19)+2(8k+25)=1(8k+17)+1(8k+23)+1(8k+29)$;
			\item $1(8k+19)+1(8k+25)+(k+1)(8k+29)=(k+3)(8k+13)+2(8k+17)$
			\item $1(8k+23)+(k+1)(8k+29)=(k+4)(8k+13)$;
			\item $2(8k+25)+(k+1)(8k+29)=(k+3)(8k+13)+1(8k+17)+1(8k+23)$.
		\end{enumerate}
	\end{lemma}
	
	Now, let us take $ B = \{0,1\} \times \{0,1\} \times \{0,1\} \times \{0,1,2\} \times \{0,1,\ldots,k\} $. Moreover, let us consider the usual product order on $\mathbb{N}^n$.
	
	\begin{theorem}\label{thm45-4}
		If $k\in\mathbb{N}$ and $S=\langle 8k+13,8k+17,8k+19,8k+23,8k+25,8k+29 \rangle$, then $\mathrm{Ap}(S,8k+13)$ is given by
		\[ \left\{ a(8k+17)+b(8k+19)+c(8k+23)+d(8k+25)+e(8k+29) \mid (a,b,c,d,e)\in C \right\}, \]
		where $ C = \{ (a,b,c,d,e) \in B \mid ab=0, \; ad=0, \; bc=0, \; cd=0, \; bd\not=2 \}  \cup \{(2,0,0,0,0)\} \cup \{(0,0,0,0,k+1),(1,0,0,0,k+1),(0,1,0,0,k+1),(0,0,0,1,k+1)\}$.
	\end{theorem}
	
	We can describe the Apéry set arranging its elements in increasing order.
	
	\begin{corollary}\label{cor46a-4}
		If $k\in\mathbb{N}\setminus\{0\}$ and $S=\langle 8k+13,8k+17,8k+19,8k+23,8k+25,8k+29 \rangle$, then
		\[ \mathrm{Ap}(S,8k+13)=\{0; (8k+29)-12,(8k+29)-10,(8k+29)-6,(8k+29)-4,8k+29; \]
		\[ 2(8k+17); 2(8k+29)-18,2(8k+29)-14,2(8k+29)-12,2(8k+29)-10, \]
		\[ 2(8k+29)-8,2(8k+29)-6,2(8k+29)-4,2(8k+29); \ldots; \]
		\[ (k+1)(8k+29)-18,(k+1)(8k+29)-14,(k+1)(8k+29)-12,(k+1)(8k+29)-10, \]
		\[ (k+1)(8k+29)-8,(k+1)(8k+29)-6,(k+1)(8k+29)-4,(k+1)(8k+29); \]
		\[ (k+2)(8k+29)-18,(k+2)(8k+29)-14,(k+2)(8k+29)-12, \]
		\[ (k+2)(8k+29)-10,(k+2)(8k+29)-8,(k+2)(8k+29)-4 \}. \]
		Moreover, if $S=\langle 13,17,19,23,25,29 \rangle$, then
		\[ \mathrm{Ap}(S,13)=\{0;17,21,23,25,29;34;40,44,46,48,50,54\}. \]
	\end{corollary}
	
	We identify a certain pattern: ``;'' separates certain groups of numbers.
	
	\begin{corollary}\label{cor46b-4}
		If $k\in\mathbb{N}$ and $S=\langle 8k+13,8k+17,8k+19,8k+23,8k+25,8k+29 \rangle$, then
		\begin{enumerate}
			\item $\mathrm{PF}(S) = \{ 8k+21,8k^2+37k+27,8k^2+37k+31,8k^2+37k+33,8k^2+37k+35,8k^2+37k+37,8k^2+37k+41 \}$;
			\item $\mathrm{F}(S) = 8k^2+37k+41$;
			\item $\mathrm{g}(S) = 4k^2+22k+27$.
		\end{enumerate}
	\end{corollary}
	
	\begin{remark}\label{casesex-4}
		From Corollary~\ref{cor46b-4}, we deduce that, if $p=8k+13$ with $k\in\mathbb{N}$, then $\mathrm{F}(p,p+4,p+6,p+10,p+12,p+16)=\frac{p^2+11p+16}{8}$.
	\end{remark}

	\subsection{Septuplets}\label{septuplets}
	
	Let us recall that a prime septuplet is of the form $(p,p+2,p+6,p+8,p+12,p+18,p+20)$ or of the form $(p,p+2,p+8,p+12,p+14,p+18,p+20)$. This fact is improved in the following proposition.
	
	\begin{proposition}\label{prop51}
		We have that:
		\begin{enumerate}
			\item If $(p,p+2,p+6,p+8,p+12,p+18,p+20)$ is a prime septuplet, then $p=210k+11$, with $k\in\mathbb{N}$.
			\item If $(p,p+2,p+8,p+12,p+14,p+18,p+20)$ is a prime septuplet, then $p=210k+179$, with $k\in\mathbb{N}$.
		\end{enumerate}
	\end{proposition}
	
	However, to study the Frobenius problem for prime septuplets, we are going to consider two cases.
	\begin{enumerate}
		\item $(p,p+2,p+6,p+8,p+12,p+18,p+20)$ for $p=10k+11$, with $k\in\mathbb{N}$;
		\item $(p,p+2,p+8,p+12,p+14,p+18,p+20)$ for $p=10k+19$, with $k\in\mathbb{N}$.
	\end{enumerate}
	Moreover, we also analyse the following cases
	\begin{enumerate}
		\item $(p,p+2,p+6,p+8,p+12,p+18,p+20)$ for $p=10k+r$, with $k\in\mathbb{N}$ and $r\in\{13,15,17,19\}$;
		\item $(p,p+2,p+8,p+12,p+14,p+18,p+20)$ for $p=10k+r$, with $k\in\mathbb{N}$ and $r\in\{11,13,15,17\}$.
	\end{enumerate}
	
	For the time being, we will present some results based on computational evidence. For this purpose, we use the \texttt{GAP} package \texttt{numericalsgps} (see \cite{numericalsgps}).
	
	\begin{evidence}\label{evid51-1}
		If $k\in\mathbb{N}\setminus\{0,1\}$ and $S=\langle p,p+2,p+6,p+8,p+12,p+18,p+20 \rangle$ with $p=10k+11$, then
		\begin{enumerate}
			\item $\mathrm{t}(S) = 13$;
			\item $\mathrm{F}(S) = 10k^2+31k+20$;
			\item $\mathrm{g}(S) = 5k^2+20k+20$.
		\end{enumerate}
		In particular, $\mathrm{F}(S) = \frac{p^2+9p-20}{10}$.
	\end{evidence}
	
	\begin{evidence}\label{evid51-2}
		If $k\in\mathbb{N}\setminus\{0,1\}$ and $S=\langle p,p+2,p+6,p+8,p+12,p+18,p+20 \rangle$ with $p=10k+13$, then
		\begin{enumerate}
			\item $\mathrm{t}(S) = 6$;
			\item $\mathrm{F}(S) = 10k^2+43k+37$;
			\item $\mathrm{g}(S) = 5k^2+22k+25$.
		\end{enumerate}
		In particular, $\mathrm{F}(S) = \frac{p^2+17p-20}{10}$.
	\end{evidence}
	
	\begin{evidence}\label{evid51-3}
		If $k\in\mathbb{N}\setminus\{0\}$ and $S=\langle p,p+2,p+6,p+8,p+12,p+18,p+20 \rangle$ with $p=10k+15$, then
		\begin{enumerate}
			\item $\mathrm{t}(S) = 7$;
			\item $\mathrm{F}(S) = 10k^2+45k+43$;
			\item $\mathrm{g}(S) = 5k^2+24k+30$.
		\end{enumerate}
		In particular, $\mathrm{F}(S) = \frac{p^2+15p-20}{10}$.
	\end{evidence}
	
	\begin{evidence}\label{evid51-4}
		If $k\in\mathbb{N}\setminus\{0\}$ and $S=\langle p,p+2,p+6,p+8,p+12,p+18,p+20 \rangle$ with $p=10k+17$, then
		\begin{enumerate}
			\item $\mathrm{t}(S) = 9$;
			\item $\mathrm{F}(S) = 10k^2+47k+49$;
			\item $\mathrm{g}(S) = 5k^2+26k+35$.
		\end{enumerate}
		In particular, $\mathrm{F}(S) = \frac{p^2+13p-20}{10}$.
	\end{evidence}
	
	\begin{evidence}\label{evid51-5}
		If $k\in\mathbb{N}\setminus\{0\}$ and $S=\langle p,p+2,p+6,p+8,p+12,p+18,p+20 \rangle$ with $p=10k+19$, then
		\begin{enumerate}
			\item $\mathrm{t}(S) = 11$;
			\item $\mathrm{F}(S) = 10k^2+49k+55$;
			\item $\mathrm{g}(S) = 5k^2+28k+40$.
		\end{enumerate}
		In particular, $\mathrm{F}(S) = \frac{p^2+11p-20}{10}$.
	\end{evidence}
	
	\begin{evidence}\label{evid52-1}
		If $k\in\mathbb{N}\setminus\{0,1\}$ and $S=\langle p,p+2,p+8,p+12,p+14,p+18,p+20 \rangle$ with $p=10k+11$, then
		\begin{enumerate}
			\item $\mathrm{t}(S) = 13$;
			\item $\mathrm{F}(S) = 10k^2+31k+20$;
			\item $\mathrm{g}(S) = 5k^2+20k+20$.
		\end{enumerate}
		In particular, $\mathrm{F}(S) = \frac{p^2+9p-20}{10}$.
	\end{evidence}
	
	\begin{evidence}\label{evid52-2}
		If $k\in\mathbb{N}$ and $S=\langle p,p+2,p+8,p+12,p+14,p+18,p+20 \rangle$ with $p=10k+13$, then
		\begin{enumerate}
			\item $\mathrm{t}(S) = 6$;
			\item $\mathrm{F}(S) = 10k^2+43k+37$;
			\item $\mathrm{g}(S) = 5k^2+22k+25$.
		\end{enumerate}
		In particular, $\mathrm{F}(S) = \frac{p^2+17p-20}{10}$.
	\end{evidence}
	
	\begin{evidence}\label{evid52-3}
		If $k\in\mathbb{N}$ and $S=\langle p,p+2,p+8,p+12,p+14,p+18,p+20 \rangle$ with $p=10k+15$, then
		\begin{enumerate}
			\item $\mathrm{t}(S) = 7$;
			\item $\mathrm{F}(S) = 10k^2+45k+43$;
			\item $\mathrm{g}(S) = 5k^2+24k+30$.
		\end{enumerate}
		In particular, $\mathrm{F}(S) = \frac{p^2+15p-20}{10}$.
	\end{evidence}
	
	\begin{evidence}\label{evid52-4}
		If $k\in\mathbb{N}$ and $S=\langle p,p+2,p+8,p+12,p+14,p+18,p+20 \rangle$ with $p=10k+17$, then
		\begin{enumerate}
			\item $\mathrm{t}(S) = 9$;
			\item $\mathrm{F}(S) = 10k^2+47k+49$;
			\item $\mathrm{g}(S) = 5k^2+26k+35$.
		\end{enumerate}
		In particular, $\mathrm{F}(S) = \frac{p^2+13p-20}{10}$.
	\end{evidence}
	
	\begin{evidence}\label{evid52-5}
		If $k\in\mathbb{N}$ and $S=\langle p,p+2,p+8,p+12,p+14,p+18,p+20 \rangle$ with $p=10k+19$, then
		\begin{enumerate}
			\item $\mathrm{t}(S) = 11$;
			\item $\mathrm{F}(S) = 10k^2+49k+55$;
			\item $\mathrm{g}(S) = 5k^2+28k+40$.
		\end{enumerate}
		In particular, $\mathrm{F}(S) = \frac{p^2+11p-20}{10}$.
	\end{evidence}
	
	Let us observe that the results for $(p,p+2,p+6,p+8,p+12,p+18,p+20)$ and $(p,p+2,p+8,p+12,p+14,p+18,p+20)$ are very similar. In fact, they are practically the same.

	\subsection{Octuplets}\label{octuplets}
	
	Let us recall that a prime octuplet is giving by one of the following forms.
	\begin{itemize}
		\item $(p,p+2,p+6,p+12,p+14,p+20,p+24,p+26)$.
		\item $(p,p+2,p+6,p+8,p+12,p+18,p+20,p+26)$.
		\item $(p,p+6,p+8,p+14,p+18,p+20,p+24,p+26)$.
	\end{itemize}
	
	\begin{proposition}\label{prop61}
		We have that:
		\begin{enumerate}
			\item If $(p,p+2,p+6,p+12,p+14,p+20,p+24,p+26)$ is a prime octuplet, then $p=30k+17$, with $k\in\mathbb{N}$.
			\item If $(p,p+2,p+6,p+8,p+12,p+18,p+20,p+26)$ is a prime octuplet, then $p=210k+11$, with $k\in\mathbb{N}$.
			\item If $(p,p+6,p+8,p+14,p+18,p+20,p+24,p+26)$ is a prime octuplet, then $p=210k+173$, with $k\in\mathbb{N}$.
		\end{enumerate}
	\end{proposition}
	
	However, to study the Frobenius problem for prime octuplets, we are going to consider the following cases.
	\begin{enumerate}
		\item $(p,p+2,p+6,p+12,p+14,p+20,p+24,p+26)$ for $p=26k+r$,
		\item $(p,p+2,p+6,p+8,p+12,p+18,p+20,p+26)$ for $p=26k+r$,
		\item $(p,p+6,p+8,p+14,p+18,p+20,p+24,p+26)$ for $p=26k+r$,
	\end{enumerate}
	with $k\in\mathbb{N}$ and $r\in\{1,3,5,7,9,11,13,15,17,19,21,23,25\}$
	
	For the time being, we will present some results based on computational evidence. For this purpose, we again use the \texttt{GAP} package \texttt{numericalsgps}.
	
	\begin{evidence}\label{evid61-1}
		If $k\in\mathbb{N}\setminus\{0,1\}$ and $S=\langle p,p+2,p+6,p+12,p+14,p+20,p+24,p+26 \rangle$ with $p=26k+1$, then
		\begin{enumerate}
			\item $\mathrm{t}(S) = 15$;
			\item $\mathrm{F}(S) = 52k^2+28k-1$;
			\item $\mathrm{g}(S) = 26k^2+26k+8$.
		\end{enumerate}
		In particular, $\mathrm{F}(S) = \frac{p^2+12p-26}{13}$.
	\end{evidence}
	
	\begin{evidence}\label{evid61-2}
		If $k\in\mathbb{N}\setminus\{0\}$ and $S=\langle p,p+2,p+6,p+12,p+14,p+20,p+24,p+26 \rangle$ with $p=26k+3$, then
		\begin{enumerate}
			\item $\mathrm{t}(S) = 6$;
			\item $\mathrm{F}(S) = 52k^2+58k+4$;
			\item $\mathrm{g}(S) = 26k^2+30k+11$.
		\end{enumerate}
		In particular, $\mathrm{F}(S) = \frac{p^2+23p-26}{13}$.
	\end{evidence}
	
	\begin{evidence}\label{evid61-3}
		If $k\in\mathbb{N}\setminus\{0\}$ and $S=\langle p,p+2,p+6,p+12,p+14,p+20,p+24,p+26 \rangle$ with $p=26k+5$, then
		\begin{enumerate}
			\item $\mathrm{t}(S) = 6$;
			\item $\mathrm{F}(S) = 52k^2+62k+8$;
			\item $\mathrm{g}(S) = 26k^2+34k+14$.
		\end{enumerate}
		In particular, $\mathrm{F}(S) = \frac{p^2+21p-26}{13}$.
	\end{evidence}
	
	\begin{evidence}\label{evid61-4}
		If $k\in\mathbb{N}\setminus\{0\}$ and $S=\langle p,p+2,p+6,p+12,p+14,p+20,p+24,p+26 \rangle$ with $p=26k+7$, then
		\begin{enumerate}
			\item $\mathrm{t}(S) = 8$;
			\item $\mathrm{F}(S) = 52k^2+66k+12$;
			\item $\mathrm{g}(S) = 26k^2+38k+17$.
		\end{enumerate}
		In particular, $\mathrm{F}(S) = \frac{p^2+19p-26}{13}$.
	\end{evidence}
	
	\begin{evidence}\label{evid61-5}
		If $k\in\mathbb{N}\setminus\{0\}$ and $S=\langle p,p+2,p+6,p+12,p+14,p+20,p+24,p+26 \rangle$ with $p=26k+9$, then
		\begin{enumerate}
			\item $\mathrm{t}(S) = 10$;
			\item $\mathrm{F}(S) = 52k^2+70k+16$;
			\item $\mathrm{g}(S) = 26k^2+42k+20$.
		\end{enumerate}
		In particular, $\mathrm{F}(S) = \frac{p^2+17p-26}{13}$.
	\end{evidence}
	
	\begin{evidence}\label{evid61-6}
		If $k\in\mathbb{N}\setminus\{0\}$ and $S=\langle p,p+2,p+6,p+12,p+14,p+20,p+24,p+26 \rangle$ with $p=26k+11$, then
		\begin{enumerate}
			\item $\mathrm{t}(S) = 12$;
			\item $\mathrm{F}(S) = 52k^2+74k+20$;
			\item $\mathrm{g}(S) = 26k^2+46k+23$.
		\end{enumerate}
		In particular, $\mathrm{F}(S) = \frac{p^2+15p-26}{13}$.
	\end{evidence}
	
	\begin{evidence}\label{evid61-7}
		If $k\in\mathbb{N}\setminus\{0\}$ and $S=\langle p,p+2,p+6,p+12,p+14,p+20,p+24,p+26 \rangle$ with $p=26k+13$, then
		\begin{enumerate}
			\item $\mathrm{t}(S) = 14$;
			\item $\mathrm{F}(S) = 52k^2+78k+24$;
			\item $\mathrm{g}(S) = 26k^2+50k+26$.
		\end{enumerate}
		In particular, $\mathrm{F}(S) = \frac{p^2+13p-26}{13}$.
	\end{evidence}
	
	\begin{evidence}\label{evid61-8}
		If $k\in\mathbb{N}\setminus\{0\}$ and $S=\langle p,p+2,p+6,p+12,p+14,p+20,p+24,p+26 \rangle$ with $p=26k+15$, then
		\begin{enumerate}
			\item $\mathrm{t}(S) = 9$;
			\item $\mathrm{F}(S) = 52k^2+108k+43$;
			\item $\mathrm{g}(S) = 26k^2+54k+30$.
		\end{enumerate}
		In particular, $\mathrm{F}(S) = \frac{p^2+24p-26}{13}$.
	\end{evidence}
	
	\begin{evidence}\label{evid61-9}
		If $k\in\mathbb{N}\setminus\{0\}$ and $S=\langle p,p+2,p+6,p+12,p+14,p+20,p+24,p+26 \rangle$ with $p=26k+17$, then
		\begin{enumerate}
			\item $\mathrm{t}(S) = 5$;
			\item $\mathrm{F}(S) = 52k^2+112k+49$;
			\item $\mathrm{g}(S) = 26k^2+58k+35$.
		\end{enumerate}
		In particular, $\mathrm{F}(S) = \frac{p^2+22p-26}{13}$.
	\end{evidence}
	
	\begin{evidence}\label{evid61-10}
		If $k\in\mathbb{N}\setminus\{0\}$ and $S=\langle p,p+2,p+6,p+12,p+14,p+20,p+24,p+26 \rangle$ with $p=26k+19$, then
		\begin{enumerate}
			\item $\mathrm{t}(S) = 7$;
			\item $\mathrm{F}(S) = 52k^2+116k+55$;
			\item $\mathrm{g}(S) = 26k^2+62k+40$.
		\end{enumerate}
		In particular, $\mathrm{F}(S) = \frac{p^2+20p-26}{13}$.
	\end{evidence}
	
	\begin{evidence}\label{evid61-11}
		If $k\in\mathbb{N}\setminus\{0\}$ and $S=\langle p,p+2,p+6,p+12,p+14,p+20,p+24,p+26 \rangle$ with $p=26k+21$, then
		\begin{enumerate}
			\item $\mathrm{t}(S) = 9$;
			\item $\mathrm{F}(S) = 52k^2+120k+61$;
			\item $\mathrm{g}(S) = 26k^2+66k+45$.
		\end{enumerate}
		In particular, $\mathrm{F}(S) = \frac{p^2+18p-26}{13}$.
	\end{evidence}
	
	\begin{evidence}\label{evid61-12}
		If $k\in\mathbb{N}\setminus\{0\}$ and $S=\langle p,p+2,p+6,p+12,p+14,p+20,p+24,p+26 \rangle$ with $p=26k+23$, then
		\begin{enumerate}
			\item $\mathrm{t}(S) = 11$;
			\item $\mathrm{F}(S) = 52k^2+124k+67$;
			\item $\mathrm{g}(S) = 26k^2+70k+50$.
		\end{enumerate}
		In particular, $\mathrm{F}(S) = \frac{p^2+16p-26}{13}$.
	\end{evidence}
	
	\begin{evidence}\label{evid61-13}
		If $k\in\mathbb{N}\setminus\{0\}$ and $S=\langle p,p+2,p+6,p+12,p+14,p+20,p+24,p+26 \rangle$ with $p=26k+25$, then
		\begin{enumerate}
			\item $\mathrm{t}(S) = 13$;
			\item $\mathrm{F}(S) = 52k^2+128k+73$;
			\item $\mathrm{g}(S) = 26k^2+74k+55$.
		\end{enumerate}
		In particular, $\mathrm{F}(S) = \frac{p^2+14p-26}{13}$.
	\end{evidence}
	
	\begin{evidence}\label{evid62-1}
		If $k\in\mathbb{N}\setminus\{0\}$ and $S=\langle p,p+2,p+6,p+8,p+12,p+18,p+20,p+26 \rangle$ with $p=26k+1$, then
		\begin{enumerate}
			\item $\mathrm{t}(S) = 7$;
			\item $\mathrm{F}(S) = 52k^2+54k-2$;
			\item $\mathrm{g}(S) = 26k^2+32k+3$.
		\end{enumerate}
		In particular, $\mathrm{F}(S) = \frac{p^2+25p-52}{13}$.
	\end{evidence}
	
	\begin{evidence}\label{evid62-2}
		If $k\in\mathbb{N}\setminus\{0\}$ and $S=\langle p,p+2,p+6,p+8,p+12,p+18,p+20,p+26 \rangle$ with $p=26k+3$, then
		\begin{enumerate}
			\item $\mathrm{t}(S) = 8$;
			\item $\mathrm{F}(S) = 52k^2+58k+4$;
			\item $\mathrm{g}(S) = 26k^2+36k+6$.
		\end{enumerate}
		In particular, $\mathrm{F}(S) = \frac{p^2+23p-26}{13}$.
	\end{evidence}
	
	\begin{evidence}\label{evid62-3}
		If $k\in\mathbb{N}\setminus\{0\}$ and $S=\langle p,p+2,p+6,p+8,p+12,p+18,p+20,p+26 \rangle$ with $p=26k+5$, then
		\begin{enumerate}
			\item $\mathrm{t}(S) = 10$;
			\item $\mathrm{F}(S) = 52k^2+62k+8$;
			\item $\mathrm{g}(S) = 26k^2+40k+9$.
		\end{enumerate}
		In particular, $\mathrm{F}(S) = \frac{p^2+21p-26}{13}$.
	\end{evidence}
	
	\begin{evidence}\label{evid62-4}
		If $k\in\mathbb{N}\setminus\{0\}$ and $S=\langle p,p+2,p+6,p+8,p+12,p+18,p+20,p+26 \rangle$ with $p=26k+7$, then
		\begin{enumerate}
			\item $\mathrm{t}(S) = 12$;
			\item $\mathrm{F}(S) = 52k^2+66k+12$;
			\item $\mathrm{g}(S) = 26k^2+44k+12$.
		\end{enumerate}
		In particular, $\mathrm{F}(S) = \frac{p^2+19p-26}{13}$.
	\end{evidence}
	
	\begin{evidence}\label{evid62-5}
		If $k\in\mathbb{N}\setminus\{0\}$ and $S=\langle p,p+2,p+6,p+8,p+12,p+18,p+20,p+26 \rangle$ with $p=26k+9$, then
		\begin{enumerate}
			\item $\mathrm{t}(S) = 7$;
			\item $\mathrm{F}(S) = 52k^2+96k+25$;
			\item $\mathrm{g}(S) = 26k^2+48k+16$.
		\end{enumerate}
		In particular, $\mathrm{F}(S) = \frac{p^2+30p-26}{13}$.
	\end{evidence}
	
	\begin{evidence}\label{evid62-6}
		If $k\in\mathbb{N}\setminus\{0\}$ and $S=\langle p,p+2,p+6,p+8,p+12,p+18,p+20,p+26 \rangle$ with $p=26k+11$, then
		\begin{enumerate}
			\item $\mathrm{t}(S) = 9$;
			\item $\mathrm{F}(S) = 52k^2+100k+27$;
			\item $\mathrm{g}(S) = 26k^2+52k+19$.
		\end{enumerate}
		In particular, $\mathrm{F}(S) = \frac{p^2+28p-78}{13}$.
	\end{evidence}
	
	\begin{evidence}\label{evid62-7}
		If $k\in\mathbb{N}$ and $S=\langle p,p+2,p+6,p+8,p+12,p+18,p+20,p+26 \rangle$ with $p=26k+13$, then
		\begin{enumerate}
			\item $\mathrm{t}(S) = 6$;
			\item $\mathrm{F}(S) = 52k^2+104k+37$;
			\item $\mathrm{g}(S) = 26k^2+56k+24$.
		\end{enumerate}
		In particular, $\mathrm{F}(S) = \frac{p^2+26p-26}{13}$.
	\end{evidence}
	
	\begin{evidence}\label{evid62-8}
		If $k\in\mathbb{N}$ and $S=\langle p,p+2,p+6,p+8,p+12,p+18,p+20,p+26 \rangle$ with $p=26k+15$, then
		\begin{enumerate}
			\item $\mathrm{t}(S) = 7$;
			\item $\mathrm{F}(S) = 52k^2+108k+43$;
			\item $\mathrm{g}(S) = 26k^2+60k+28$.
		\end{enumerate}
		In particular, $\mathrm{F}(S) = \frac{p^2+24p-26}{13}$.
	\end{evidence}
	
	\begin{evidence}\label{evid62-9}
		If $k\in\mathbb{N}$ and $S=\langle p,p+2,p+6,p+8,p+12,p+18,p+20,p+26 \rangle$ with $p=26k+17$, then
		\begin{enumerate}
			\item $\mathrm{t}(S) = 9$;
			\item $\mathrm{F}(S) = 52k^2+112k+49$;
			\item $\mathrm{g}(S) = 26k^2+64k+33$.
		\end{enumerate}
		In particular, $\mathrm{F}(S) = \frac{p^2+22p-26}{13}$.
	\end{evidence}
	
	\begin{evidence}\label{evid62-10}
		If $k\in\mathbb{N}$ and $S=\langle p,p+2,p+6,p+8,p+12,p+18,p+20,p+26 \rangle$ with $p=26k+19$, then
		\begin{enumerate}
			\item $\mathrm{t}(S) = 11$;
			\item $\mathrm{F}(S) = 52k^2+116k+55$;
			\item $\mathrm{g}(S) = 26k^2+68k+38$.
		\end{enumerate}
		In particular, $\mathrm{F}(S) = \frac{p^2+20p-26}{13}$.
	\end{evidence}
	
	\begin{evidence}\label{evid62-11}
		If $k\in\mathbb{N}$ and $S=\langle p,p+2,p+6,p+8,p+12,p+18,p+20,p+26 \rangle$ with $p=26k+21$, then
		\begin{enumerate}
			\item $\mathrm{t}(S) = 9$;
			\item $\mathrm{F}(S) = 52k^2+120k+61$;
			\item $\mathrm{g}(S) = 26k^2+72k+43$.
		\end{enumerate}
		In particular, $\mathrm{F}(S) = \frac{p^2+18p-26}{13}$.
	\end{evidence}
	
	\begin{evidence}\label{evid62-12}
		If $k\in\mathbb{N}$ and $S=\langle p,p+2,p+6,p+8,p+12,p+18,p+20,p+26 \rangle$ with $p=26k+23$, then
		\begin{enumerate}
			\item $\mathrm{t}(S) = 8$;
			\item $\mathrm{F}(S) = 52k^2+150k+88$;
			\item $\mathrm{g}(S) = 26k^2+76k+49$.
		\end{enumerate}
		In particular, $\mathrm{F}(S) = \frac{p^2+29p-52}{13}$.
	\end{evidence}
	
	\begin{evidence}\label{evid62-13}
		If $k\in\mathbb{N}$ and $S=\langle p,p+2,p+6,p+8,p+12,p+18,p+20,p+26 \rangle$ with $p=26k+25$, then
		\begin{enumerate}
			\item $\mathrm{t}(S) = 8$;
			\item $\mathrm{F}(S) = 52k^2+154k+98$;
			\item $\mathrm{g}(S) = 26k^2+80k+55$.
		\end{enumerate}
		In particular, $\mathrm{F}(S) = \frac{p^2+27p-26}{13}$.
	\end{evidence}
	
	\begin{evidence}\label{evid63-1}
		If $k\in\mathbb{N}\setminus\{0\}$ and $S=\langle p,p+6,p+8,p+14,p+18,p+20,p+24,p+26 \rangle$ with $p=26k+1$, then
		\begin{enumerate}
			\item $\mathrm{t}(S) = 7$;
			\item $\mathrm{F}(S) = 52k^2+54k+12$;
			\item $\mathrm{g}(S) = 26k^2+32k+9$.
		\end{enumerate}
		In particular, $\mathrm{F}(S) = \frac{p^2+25p+130}{13}$.
	\end{evidence}
	
	\begin{evidence}\label{evid63-2}
		If $k\in\mathbb{N}\setminus\{0\}$ and $S=\langle p,p+6,p+8,p+14,p+18,p+20,p+24,p+26 \rangle$ with $p=26k+3$, then
		\begin{enumerate}
			\item $\mathrm{t}(S) = 8$;
			\item $\mathrm{F}(S) = 52k^2+58k+16$;
			\item $\mathrm{g}(S) = 26k^2+36k+12$.
		\end{enumerate}
		In particular, $\mathrm{F}(S) = \frac{p^2+23p+130}{13}$.
	\end{evidence}
	
	\begin{evidence}\label{evid63-3}
		If $k\in\mathbb{N}\setminus\{0\}$ and $S=\langle p,p+6,p+8,p+14,p+18,p+20,p+24,p+26 \rangle$ with $p=26k+5$, then
		\begin{enumerate}
			\item $\mathrm{t}(S) = 10$;
			\item $\mathrm{F}(S) = 52k^2+62k+20$;
			\item $\mathrm{g}(S) = 26k^2+40k+15$.
		\end{enumerate}
		In particular, $\mathrm{F}(S) = \frac{p^2+21p+130}{13}$.
	\end{evidence}
	
	\begin{evidence}\label{evid63-4}
		If $k\in\mathbb{N}\setminus\{0\}$ and $S=\langle p,p+6,p+8,p+14,p+18,p+20,p+24,p+26 \rangle$ with $p=26k+7$, then
		\begin{enumerate}
			\item $\mathrm{t}(S) = 12$;
			\item $\mathrm{F}(S) = 52k^2+66k+24$;
			\item $\mathrm{g}(S) = 26k^2+44k+18$.
		\end{enumerate}
		In particular, $\mathrm{F}(S) = \frac{p^2+19p+130}{13}$.
	\end{evidence}
	
	\begin{evidence}\label{evid63-5}
		If $k\in\mathbb{N}\setminus\{0\}$ and $S=\langle p,p+6,p+8,p+14,p+18,p+20,p+24,p+26 \rangle$ with $p=26k+9$, then
		\begin{enumerate}
			\item $\mathrm{t}(S) = 7$;
			\item $\mathrm{F}(S) = 52k^2+96k+37$;
			\item $\mathrm{g}(S) = 26k^2+48k+22$.
		\end{enumerate}
		In particular, $\mathrm{F}(S) = \frac{p^2+30p+130}{13}$.
	\end{evidence}
	
	\begin{evidence}\label{evid63-6}
		If $k\in\mathbb{N}\setminus\{0\}$ and $S=\langle p,p+6,p+8,p+14,p+18,p+20,p+24,p+26 \rangle$ with $p=26k+11$, then
		\begin{enumerate}
			\item $\mathrm{t}(S) = 9$;
			\item $\mathrm{F}(S) = 52k^2+100k+43$;
			\item $\mathrm{g}(S) = 26k^2+52k+25$.
		\end{enumerate}
		In particular, $\mathrm{F}(S) = \frac{p^2+28p+130}{13}$.
	\end{evidence}
	
	\begin{evidence}\label{evid63-7}
		If $k\in\mathbb{N}$ and $S=\langle p,p+6,p+8,p+14,p+18,p+20,p+24,p+26 \rangle$ with $p=26k+13$, then
		\begin{enumerate}
			\item $\mathrm{t}(S) = 6$;
			\item $\mathrm{F}(S) = 52k^2+104k+49$;
			\item $\mathrm{g}(S) = 26k^2+56k+30$.
		\end{enumerate}
		In particular, $\mathrm{F}(S) = \frac{p^2+26p+130}{13}$.
	\end{evidence}
	
	\begin{evidence}\label{evid63-8}
		If $k\in\mathbb{N}$ and $S=\langle p,p+6,p+8,p+14,p+18,p+20,p+24,p+26 \rangle$ with $p=26k+15$, then
		\begin{enumerate}
			\item $\mathrm{t}(S) = 7$;
			\item $\mathrm{F}(S) = 52k^2+108k+55$;
			\item $\mathrm{g}(S) = 26k^2+60k+34$.
		\end{enumerate}
		In particular, $\mathrm{F}(S) = \frac{p^2+24p+130}{13}$.
	\end{evidence}
	
	\begin{evidence}\label{evid63-9}
		If $k\in\mathbb{N}$ and $S=\langle p,p+6,p+8,p+14,p+18,p+20,p+24,p+26 \rangle$ with $p=26k+17$, then
		\begin{enumerate}
			\item $\mathrm{t}(S) = 9$;
			\item $\mathrm{F}(S) = 52k^2+112k+61$;
			\item $\mathrm{g}(S) = 26k^2+64k+39$.
		\end{enumerate}
		In particular, $\mathrm{F}(S) = \frac{p^2+22p+130}{13}$.
	\end{evidence}
	
	\begin{evidence}\label{evid63-10}
		If $k\in\mathbb{N}$ and $S=\langle p,p+6,p+8,p+14,p+18,p+20,p+24,p+26 \rangle$ with $p=26k+19$, then
		\begin{enumerate}
			\item $\mathrm{t}(S) = 11$;
			\item $\mathrm{F}(S) = 52k^2+116k+67$;
			\item $\mathrm{g}(S) = 26k^2+68k+44$.
		\end{enumerate}
		In particular, $\mathrm{F}(S) = \frac{p^2+20p+130}{13}$.
	\end{evidence}
	
	\begin{evidence}\label{evid63-11}
		If $k\in\mathbb{N}$ and $S=\langle p,p+6,p+8,p+14,p+18,p+20,p+24,p+26 \rangle$ with $p=26k+21$, then
		\begin{enumerate}
			\item $\mathrm{t}(S) = 13$;
			\item $\mathrm{F}(S) = 52k^2+120k+73$;
			\item $\mathrm{g}(S) = 26k^2+72k+49$.
		\end{enumerate}
		In particular, $\mathrm{F}(S) = \frac{p^2+18p+130}{13}$.
	\end{evidence}
	
	\begin{evidence}\label{evid63-12}
		If $k\in\mathbb{N}$ and $S=\langle p,p+6,p+8,p+14,p+18,p+20,p+24,p+26 \rangle$ with $p=26k+23$, then
		\begin{enumerate}
			\item $\mathrm{t}(S) = 8$;
			\item $\mathrm{F}(S) = 52k^2+150k+102$;
			\item $\mathrm{g}(S) = 26k^2+76k+55$.
		\end{enumerate}
		In particular, $\mathrm{F}(S) = \frac{p^2+29p+130}{13}$.
	\end{evidence}
	
	\begin{evidence}\label{evid63-13}
		If $k\in\mathbb{N}$ and $S=\langle p,p+6,p+8,p+14,p+18,p+20,p+24,p+26 \rangle$ with $p=26k+25$, then
		\begin{enumerate}
			\item $\mathrm{t}(S) = 8$;
			\item $\mathrm{F}(S) = 52k^2+154k+110$;
			\item $\mathrm{g}(S) = 26k^2+80k+61$.
		\end{enumerate}
		In particular, $\mathrm{F}(S) = \frac{p^2+27p+130}{13}$.
	\end{evidence}

	\section*{Acknowledgement}
	
	Both authors are supported by the project MTM2017-84890-P (funded by Mi\-nis\-terio de Econom\'{\i}a, Industria y Competitividad and Fondo Europeo de Desarrollo Regional FEDER) and by the Junta de Andaluc\'{\i}a Grant Number FQM-343.


\begin{thebibliography}{00}
		
		\bibitem{apery} R.~Ap\'ery, Sur les branches superlin\'{e}aires des courbes alg\'{e}briques, \textit{C.~R. Acad. Sci. Paris} \textbf{222} (1946), 1198--1200.
		
		\bibitem{barucci} V.~Barucci, D.~E.~Dobbs, and M.~Fontana, \textit{Maximality Properties in Numerical Semigroups and Applications to One-Dimensional Analytically Irreducible Local Domains}, Mem. Amer. Math. Soc., vol. 125, no. 598 (Amer. Math. Soc., Providence, RI, 1997).
		
		\bibitem{brauer} A.~Brauer and J.~E.~Shockley, On a problem of Frobenius, \textit{J. Reine Angew. Math.} \textbf{211} (1962), 215--220.
		
		\bibitem{curtis} F.~Curtis, On formulas for the Frobenius number of a numerical semigroup, \textit{Math. Scand.} \textbf{67} (1990), 190--192.
			
		\bibitem{numericalsgps} M. Delgado, P.~A. Garc{\'i}a-S\'{a}nchez, and J. Morais, NumericalSgps, a GAP package for numerical semigroups, version 1.2.2 (03/03/2020). https://gap-packages.github.io/numericalsgps
		
		\bibitem{erdos-riesel} P.~Erd\H{o}s and H.~Riesel, On admissible constellations of consecutive primes, \textit{BIT} \textbf{28} (1988), 391--396.
		
		\bibitem{forbes} T.~Forbes, Prime clusters and Cunningham chains, \textit{Math. Comp.} \textbf{68}(228), (1999), 1739--1747.
		
		\bibitem{froberg} R.~Fr\"{o}berg, G.~Gottlieb, and R.~H\"{a}ggkvist, On numerical semigroups, \textit{Semigroup Forum} \textbf{35} (1987), 63--83.
		
		\bibitem{hardy-littlewood} G.~H.~Hardy and J.~E.~Littlewood, Some problems of `partitio numerorum'; III: on the expression of a number as a sum of primes, \textit{Acta Math.} \textbf{44} (1923), 1--70.
		
		\bibitem{hardy-wright} G.~H.~Hardy and E.~M.~Wright, \textit{An Introduction to the Theory of Numbers, 5th edition} (Oxford Univ. Press, Oxford, 1979).
		
		\bibitem{komatsu1} T.~Komatsu, Sylvester sums on the Frobenius set in arithmetic progression, in: F.~Yilmaz et al. (Eds.), \textit{Mathematical Methods for Engineering Applications. ICMASE 2021.} Springer Proceedings in Mathematics \& Statistics, vol. \textbf{384} (Springer, Cham, 2022), pp. 1--23.
		
		\bibitem{komatsu2} T.~Komatsu, Sylvester sums on the Frobenius set in arithmetic progression with initial gaps, in: P. Debnath et al. (Eds.), \textit{Advances in Number Theory and Applied Analysis}, Springer Nature. In press.
		
		\bibitem{komatsu3} T.~Komatsu, The Frobenius number for Fibonacci triplet associated with number of	representations, preprint (arXiv:2206.05660v2 [math.CO]).
					
		\bibitem{alfonsin} J.~L.~Ram\'{\i}rez Alfons\'{\i}n, \textit{The Diophantine Frobenius Problem}, Oxford Lectures Series in Mathematics and its Applications, vol. 30 (Oxford Univ. Press, Oxford, 2005).
		
		
		\bibitem{JPAA} J.~C.~Rosales and M.~B.~Branco, Numerical semigroups that can be expressed as an intersection of symmetric numerical semigroups, \textit{J. Pure Appl. Algebra} \textbf{171} (2002), 303--314.
		
		\bibitem{springer} J.~C.~Rosales and P.~A.~Garc\'{\i}a-S\'anchez, \textit{Numerical Semigroups}, Developments in Mathematics, vol. 20 (Springer, New York, 2009).
		
		\bibitem{selmer} E.~S.~Selmer, On the linear diophantine problem of Frobenius, \textit{J. Reine Angew. Math.} \textbf{293/294} (1977), 1--17.
		
		\bibitem{sylvester} J.~J.~Sylvester, Problem 7382, \textit{The Educational Times, and Journal of the College of Preceptors, New Ser.}, \textbf{36}(266) (1883), 177. Solution by W.~J.~Curran Sharp, ibid., \textbf{36}(271) (1883), 315.
		
	\end{thebibliography}
\end{document}